\documentclass[12pt]{amsart}

\usepackage{xcolor}
\usepackage{enumitem}

\usepackage{amsmath,amssymb,setspace,nicefrac, yhmath, amscd,eucal}
\usepackage[active]{srcltx}
\usepackage[pagebackref, colorlinks, linkcolor=red, citecolor=blue, urlcolor=blue, hypertexnames=true]{hyperref}
\usepackage{amsrefs}
\usepackage[all,cmtip]{xy}
\allowdisplaybreaks
\usepackage[all]{xy}

\newcommand{\Z}{{\mathbb Z}}

\newtheorem{thm}{Theorem}[section]
\newtheorem{cor}[thm]{Corollary}
\newtheorem{lemma}[thm]{Lemma}

\newtheorem{definition}[thm]{Definition}
\newtheorem{example}[thm]{Example}
\newtheorem{remark}[thm]{Remark}

\newtheorem{proposition}[thm]{Proposition}

\theoremstyle{definition}

\begin{document}
\title[Tensor products and Intermediate Factor theorem]{Splitting of Tensor Products and Intermediate Factor theorem: Continuous Version}
\author[Amrutam]{Tattwamasi Amrutam}
\address{Ben Gurion University of the Negev.
Department of Mathematics.
Be'er Sheva, 8410501, Israel.}
\address{Institute of Mathematics of the Polish Academy of Sciences, Ul. S'niadeckich 8, 00-656 Warszawa, Poland}
\email{tattwamasiamrutam@gmail.com}
\author[Jiang]{Yongle Jiang}
\address{School of Mathematical Sciences, Dalian University of Technology, Dalian, 116024, China}
\email{yonglejiang@dlut.edu.cn}
\subjclass[2020]{Primary 46M05, 46L06 ; Secondary 37B05, 46L45}
\keywords{Minimal tensor products, Intermediate subalgebras,
Intermediate Factor theorem, Splitting of tensor products}
\date{\today}
\begin{abstract}
Let $G$ be a discrete group. Given unital $G$-$C^*$-algebras $\mathcal{A}$ and $\mathcal{B}$, we give an abstract condition under which every $G$-subalgebra $\mathcal{C}$ of the form $\mathcal{A}\subset \mathcal{C}\subset \mathcal{A}\otimes_{\text{min}}\mathcal{B}$ is a tensor product. This generalizes the well-known splitting results in the context of $C^*$-algebras by Zacharias and Zsido. As an application, we prove a topological version of the Intermediate Factor theorem. When a product group $G=\Gamma_1\times\Gamma_2$ acts (by a product action) on the product of corresponding $\Gamma_i$-boundaries $\partial\Gamma_i$, using the abstract condition, we show that every intermediate subalgebra $C(X)\subset\mathcal{C}\subset C(X)\otimes_{\text{min}}C(\partial\Gamma_1\times \partial\Gamma_2)$ is a tensor product (under some additional assumptions on $X$). This can be considered as a topological version of the Intermediate Factor theorem. We prove that our assumptions are necessary and cannot generally be relaxed. We also introduce the notion of a  uniformly rigid action for $C^*$-algebras and use it to give various classes of inclusions $\mathcal{A}\subset \mathcal{A}\otimes_{\text{min}}\mathcal{B}$ for which every invariant intermediate algebra is a tensor product.     \end{abstract}
\maketitle

\tableofcontents
\section{Introduction}
A basic question in the context of tensor product of $C^*$ (or, von Neumann)-algebras is to determine when a subalgebra $\mathcal{A}\subset \mathcal{C}\subset \mathcal{A}\otimes_{\text{min}}\mathcal{B}$ is a tensor product again. From the operator algebraic framework, this question has been studied by Ge-Kadison~\cite{ge1996tensor} (also see \cite{struatilua1999commutation}) in the context of von Neumann algebras and later, by \cite{zacharias2001splitting} and \cite{zsido2000criterion} for $C^*$-algebras. Seen from the lens of dynamics, if the ambient algebra is a tensor product of two abelian parts, then the question becomes equivalent to determining when every intermediate factor $Z$ with
\[X\times Y\to Z\to X\]
is again of the form $X\times \Tilde{Y}$ for a factor $\Tilde{Y}$ of $Y$. Of course, in the measurable context, we have to bring the measures into consideration.

Such results, known as Intermediate Factor theorems (IFTs), have been at the heart of numerous rigidity-type results (see, for example, \cite {bader2006factor} and the references therein). Motivated by these results, we prove an IFT for the continuous action of product groups in this paper.
\begin{thm}
\label{thm:introcontinuousfactortheorem}
Let $G=\Gamma_1\times \Gamma_2$ be a product of two discrete groups. Let $\partial\Gamma_i$ be a $\Gamma_i$-boundary for each $i=1,2$. Consider the product action $G\curvearrowright \partial\Gamma_1\times \partial\Gamma_2$. Let $X$ be a minimal $\Gamma_i$-space for each $i=1,2$. Moreover, assume that $\Gamma_i\curvearrowright X\times \partial\Gamma_i$ is minimal for each $i=1,2$. Then, every $G$-invariant intermediate $C^*$-subalgebra $\mathcal{C}$ with
\[C(X)\subset\mathcal{C}\subset C(X)\otimes_{\text{min}} C(\partial\Gamma_1\times \partial\Gamma_2)\]
is a tensor product of the form $ C(X)\otimes_{\text{min}} C(B)$, where $B$ is a $G$-factor of $\partial\Gamma_1\times \partial\Gamma_2$ which splits as a product of a $\Gamma_1$-factor of $\partial\Gamma_1$ and a $\Gamma_2$-factor of $\partial \Gamma_2$. In particular, 
every $G$-factor of $\partial\Gamma_1\times \partial\Gamma_2$ always splits as a product $Y_1\times Y_2$, where $Y_i$ is a $\Gamma_i$-factor of $\partial\Gamma_i$ (by taking $X$ to be the one point space).
\end{thm}
It is worth pointing out that the assumption of minimality of the action $\Gamma_i\curvearrowright X\times \partial\Gamma_i$ is natural to this setup (see the discussion in Subsection~\ref{subsec:IFT}).
Moreover, we can only require just one action $\Gamma_i\curvearrowright \partial\Gamma_i$ to be a $\Gamma_i$-boundary besides satisfying all other assumptions, see Theorem \ref{thm:continuousfactortheorem}. Moreover, it is natural to ask as to why we do not consider the Furstenberg boundary $\partial_FG=\partial_F(\Gamma_1\times\Gamma_2)$ associated with the product group $G=\Gamma_1\times\Gamma_2$ in the above theorem. While this would be the natural analog to the measurable version obtained in \cite[Theorem~1.9]{bader2006factor}, it is simply not true that $\partial_F(\Gamma_1\times\Gamma_2)=\partial_F\Gamma_1\times\partial_F\Gamma_2$ (see Proposition~\ref{prop:notproductboundary}).

The intermediate factor theorem (IFT) proved in \cite{bader2006factor} has a dynamical flavor and considers measurable properties of the Poisson boundary. From the operator algebraic perspective, as mentioned above, an IFT is the setup where every intermediate subalgebra is a tensor product. Our point of view lies in putting these two perspectives together. As such, we generalize the result obtained in \cite{zacharias2001splitting} to take into the group actions. In particular, we establish the following, one of our main tools in proving Theorem~\ref{thm:introcontinuousfactortheorem}. 
\begin{proposition}
\label{prop:intromainsplittingC*}
Let $\mathcal{A}$ be a $G$-simple unital $C^*$-algebra. Assume that $\mathcal{A}$ has property (S). Let $\mathcal{B}$ be a $G$-$C^*$-algebra and consider the minimal tensor product $\mathcal{A}\otimes_{\text{min}} \mathcal{B}$. Let $\mathcal{C}$ be an intermediate $G$-$C^*$-algebra of the form $\mathcal{A}\otimes_{\text{min}} \mathbb{C}\subset \mathcal{C}\subset \mathcal{A }\otimes_{\text{min}} \mathcal{B}$. Assume that there exists a subgroup $\Gamma\le  G$ such that $\mathcal{A}$ is $\Gamma$-simple and $\Gamma\curvearrowright \mathcal{B}$ is trivial.  Then, $\mathcal{C}$ is of the form $\mathcal{A}\otimes_{\text{min}} \Tilde{B}$, where $\Tilde{\mathcal{B}}$ is a $G$-$C^*$-subalgebra of $\mathcal{B}$.   \end{proposition}
We remark that this is a building block towards proving Theorem~\ref{thm:introcontinuousfactortheorem}. Even though the proof of this is a mere modification of arguments in \cite{zacharias2001splitting}, it allows us to conclude splitting results for numerous inclusions.

Property (S) is a natural assumption to have on $C^*$-algebras in this context (see Subsection~\ref{subsec:propertyS} for the definition). If we take $G=\{e\}$, we get back the original splitting theorem obtained in \cite{zacharias2001splitting} (or, \cite{zsido2000criterion}). Moreover, under additional assumptions on the group $G$, we get a complete description of intermediate subalgebras for a wide class of actions. This is the case, for example, when $G=\Gamma_1\times\Gamma_2$ for a $C^*$-simple group $\Gamma_1$ and an amenable group $\Gamma_2$. More generally, we have the following.
\begin{cor}
\label{cor:C*simpleandamenable}
Let $G$ be a non-amenable group with a nontrivial amenable radical $R_a(G)$. Let $\partial G$ be any $G$-boundary. Let $\mathcal{A}$ be an $H$-simple unital $C^*$-algebra with property (S), where $H\le R_a(G)$. Then, every $G$-invariant intermediate $C^*$-algebra $\mathcal{C}$ of the form
\[\mathcal{A}\subset\mathcal{C}\subset \mathcal{A}\otimes_{\text{min}} C(\partial G)\]
splits.
\end{cor}
At the same time,  we also prove a splitting theorem for the cases when $\mathcal{A}$ is a reduced group $C^*$-algebra for a wide class of groups $\Gamma$ acting non-faithfully on a unital $C^*$-algebra $\mathcal{B}$. To be precise, we show the following.
\begin{cor}
\label{cor:introsplittingacylindricallyhyperbolic}
Let $G$ be a $C^*$-simple acylindrically hyperbolic group with property-AP. Let $\mathcal{B}$ be a unital $G$-$C^*$-algebra such that $G\curvearrowright\mathcal{B}$ is non-faithful. Then, every $G$-invariant intermediate $C^*$-subalgebra $\mathcal{C}$ with
\[C_r^*(G)\otimes_{\text{min}} \mathbb{C}\subset\mathcal{C}\subset C_r^*(G)\otimes_{\text{min}}\mathcal{B}\]
is a tensor product.
\end{cor}
The instances of acylindrically hyperbolic groups encompass all non-(virtually) cyclic groups that are hyperbolic relative to proper subgroups, $\text{Out}(F_n)$ for $n > 1$, nearly all mapping class groups with only a few exceptions, groups that are not virtually cyclic and act appropriately on proper CAT($0$)-spaces while also containing elements of rank one, among others. For further details, we refer the readers to \cite{AO} and the citations provided within.

Moreover, using the recently introduced notion of rigid strongly proximal sequences~\cite{AGG}, we prove an IFT for uniformly rigid actions of non-elementary convergence groups and lattices in $SL_d(\mathbb{R})$. In these cases, the underlying space has to be either the limit set of the non-elementary convergence group or the Grassman variety of $l$-dimensional subspaces of $\mathbb{R}^d$ induced by the projective linear transformations on the projective space.
\begin{thm}
\label{thm: convergencegroupandlattices}
\
\begin{enumerate}
    \item Let $G$ be a non-elementary convergence group. Let $G\curvearrowright X$ be a uniformly rigid continuous action and $G\curvearrowright Y$ be the limit set for a non-elementary convergence action of $G$. Then every  $G$-invariant intermediate $C^*$-algebra $Q$ with $C(Y)\subset Q\subset C(X)\otimes_{\text{min}} C(Y)$ splits as a tensor product.
    \item  Let $\Gamma<G=SL_d(\mathbb{R})$, $d\geq 2$, be a lattice. Let $\Gamma\curvearrowright X$ be a uniformly rigid action with a fixed, rigid sequence $\{s_n: n\in\mathbb{N}\}$. Then there exists some $0\leq\ell\leq  n-1$ such that if we set  $\Lambda\curvearrowright Y$ be the action on $Y=Gr(\ell, \mathbb{R}^d)$, the Grassmann variety of $\ell$-dimensional subspaces of $\mathbb{R}^d$ induced by the projective linear transformations of $\Gamma$ on $\mathbb{P}^{d-1}(\mathbb{R})$, the projective space of all lines in $\mathbb{R}^d$. Then every $\Gamma$-invariant intermediate $C^*$-subalgebra $Q$ between $C(Y)$ and $C(X)\otimes_{\text{min}} C(Y)$ splits as a tensor product.
\end{enumerate}
\end{thm}
The notion of rigid actions for topological actions was introduced by Glasner-Maon~\cite{glasner1989rigidity}. Motivated by this, we introduce the notion of a uniformly rigid action for a unital $\Gamma$-$C^*$-algebra $\mathcal{A}$, not necessarily commutative (see Definition~\ref{def: strongly uniformly rigid}). It agrees with the notion established in the topological dynamics (see Proposition~\ref{prop: ur is weaker than sur}). Moreover, there are many natural classes of actions on $C^*$-algebras that fall into our mold, for example, $\mathcal{A}$ being a finite-dimensional $C^*$-algebra and the action $G\curvearrowright\mathcal{A}$ implemented by $*$-automorphisms. Punting the boundary actions against uniformly rigid actions, we prove the following.
\begin{thm}\label{thm:introsplitting using RSP sequence in the commutative times noncommutative setting}
Let $\beta: G\curvearrowright \mathcal{A}$ and $\alpha: G\curvearrowright Y$ be two continuous actions, where $\mathcal{A}$ is a $C^*$-algebra and $Y$ is a compact Hausdorff space.
Assume the following conditions hold:
\begin{itemize}
    \item $\alpha: G\curvearrowright Y$ is a $G$-boundary, i.e., it is a strongly proximal minimal $G$-action;
    \item $\beta: G\curvearrowright \mathcal{A}$ is a uniformly rigid action;
    \item There exists an $(\alpha,\beta)$-RSP sequence $\{s_i\}\subset G$.
\end{itemize}
Then every $G$-invariant intermediate $C^*$-algebra $Q$ with $C(Y)\otimes_{\text{min}}\mathbb{C}\subset Q\subset C(Y)\otimes_{\text{min}} \mathcal{A}$ splits as a tensor product.
\end{thm}
Finally, we want to mention that the works \cites{bader2006factor,glasner2023intermediate} influenced us to pursue this direction. In a subsequent paper, we will focus on the measurable setup. Note that we frequently use product and diagonal actions in this paper; one should be careful of the difference.

\subsection*{Organization of the paper} In addition to this section, there are six other chapters in this paper. In Section~\ref{sec:preliminaries}, we put together some basic facts and results, which are used later. We prove Proposition~\ref{prop:intromainsplittingC*} in Section~\ref{sec:simpleinvariantC*-algebras}. While we are here, we take care of Corollary~\ref{cor:C*simpleandamenable} and Corollary~\ref{cor:introsplittingacylindricallyhyperbolic}. In Section~\ref{sec:intermediate}, we prove Theorem~\ref{thm:introcontinuousfactortheorem}. We discuss about the assumptions made in our splitting results in Section~\ref{sec:simplicityassumption}. In particular, we show why they are necessary (see Proposition~\ref{prop:condforsplitting} and Theorem~\ref{thm:hom}). Finally, in Section~\ref{sec:morenoncommutative}, we introduce the notion of uniformly rigid actions and prove Theorem~\ref{thm:introsplitting using RSP sequence in the commutative times noncommutative setting}. 
\subsection*{Acknowledgements} The first named author thanks Mehrdad Kalantar and Yair Hartman for numerous discussions about the intermediate factor theorems. We thank Hanna Oppelmayer for many helpful conversations. We also thank Hanfeng Li and Adam Skalski for several insightful comments and necessary corrections. We also thank the anonymous reviewers for their numerous comments and corrections. Y. J. is partially supported by the National Natural Science Foundation of China (Grant No. 12001081). T. A. is partially supported by the European Research Council (ERC) under the European Union’s Seventh Framework Program (FP7-2007-2013) (Grant agreement No.101078193) and Israel Science Foundation under the grant
ISF 1175/18.
\section{Preliminaries and Notation}
\label{sec:preliminaries}
Throughout this paper, we shall only consider discrete groups denoted by $G$. All the $C^*$-algebras considered will be unital. A unital $C^*$-algebra $\mathcal{A}$ on which $G$ acts by $*$-automorphisms will be called a $G$-$C^*$-algebra. $\mathcal{A}$ is called $G$-simple if the only $G$-invariant closed two sided ideals in $\mathcal{A}$ are either $(0)$ or $\mathcal{A}$.
Given a $G$-$C^*$-algebra $\mathcal{A}$ and a $G$-space $X$, i.e., a continuous action $G\curvearrowright X$ on a compact Hausdorff space $X$, we form the minimal tensor product $C(X)\otimes_{\text{min}} \mathcal{A}$, which is still a $G$-$C^*$-algebra with respect to the diagonal $G$-action. We sometimes identify $C(X)\otimes_{\text{min}} \mathcal{A}$ with $C(X,\mathcal{A})$ (see \cite[Page 849]{KR2}). We interchangeably use this identification without explicitly mentioning it. We are interested in determining the structure of $G$-invariant $C^*$-subalgebras $Q$ of the form
\[C(X)\subset Q\subset C(X)\otimes_{\text{min}}\mathcal{A}.\]
We call such subalgebras as intermediate $G$-$C^*$-subalgebras. Whenever $Q$ is a tensor product of the form $C(X)\otimes_{\text{min}}\mathcal{B}$ for some $G$-invariant subalgebra $\mathcal{B}$ of $\mathcal{A}$, we say that $Q$ splits. It is well-known that $Q$ can be written as a direct integral of subalgebras of $\mathcal{A}$ over $X$ (an analogous result in the von Neumann algebraic setup appears in \cite[Proposition 14.1.18]{KR2}). We include the proof for the sake of completion.
\begin{lemma}\label{lem:smallparts}
Let $X$ be a compact Hausdorff space and $\mathcal{A}$ be a $C^*$-algebra. Let $C(X)\subset Q\subset C(X)\otimes_{\text{min}} \mathcal{A}$ be an intermediate $C^*$-subalgebra. For any $x\in X$, set $Q_x=\{f(x)\mid f\in Q\}\subset \mathcal{A}$, where we have identified $C(X)\otimes_{\text{min}} \mathcal{A}$ with $C(X, \mathcal{A})$. Then $Q=\int_{x\in X}Q_x:=\{f\in C(X, \mathcal{A})\mid f(x)\in Q_x\}$.
\end{lemma}

\begin{proof}
For any $x\in X$, note that $\pi_x: C(X,\mathcal{A})\rightarrow \mathcal{A}$ defined by $\pi_x(f)=f(x)$ is a *-homomorphism between $C^*$-algebras. Hence $Q_x=\pi_x(Q)\subset \mathcal{A}$ is a $C^*$-subalgebra by \cite[Theorem 4.1.9]{KR1}.
By definition of $Q_x$, it is also clear that $Q\subset \int_{x\in X}Q_x$.
We are left to show the converse inclusion also holds.

Fix any $0<\epsilon<1$ and any $f\in \int_{x\in X}Q_x$. For any $x\in X$, there is some $f_x\in Q$ such that $f_x(x)=f(x)$ since $f(x)\in Q_x$. Since $f_x$ is continuous, there is some small open neighborhood $U_x\ni x$ such that
\begin{equation}\label{eq:newineq}
\left\|f_x(y)-f(y)\right \|<\epsilon,~\forall y\in U_x.\end{equation}
Since $X$ is compact, we can find finitely many points $x_1,\ldots, x_n$ in $X$ such that $X=\cup_{i=1}^nU_{x_i}$. Let $\{h_i\}_{i=1}^n$ be a partition of unity subordinate to $\{U_{x_i}\}_{i=1}^n$, i.e., $0\leq h_i\in C(X)\subset Q$, $\text{supp}(h_i)\subset U_{x_i}$ and $\sum_{i=1}^nh_i\equiv 1$.

Note that $a:=\sum_{i=1}^nh_if_{x_i}\in Q$ and let us prove that $||a-f||<\epsilon$.

Let $y\in X$. Let $i_1, i_2, \ldots,i_k\subseteq \{1,2,\ldots,n\}$ be such that $y\not\in U_{x_{i}}$ iff $i\not\in\{i_1,i_2,\ldots,i_k\}$. Then, it follows that $h_{i}(y)= 0$ for $i\not\in\{i_1,i_2,\ldots,i_k\}$. Moreover, we have that $\left\|f_{x_{i_j}}(y)-f_{x_{i_j}}(x_{i_j})\right\|<\epsilon$ and $\sum_{i=1}^nh_i(y)=\sum_{j=1}^kh_{i_j}(y)=1.$

For any $y\in X$, we have that
\begin{align*}
\left\|a(y)-f(y)\right\|&=\left\|\sum_{i=1}^nh_i(y)f_{x_i}(y)-\sum_{i=1}^nh_i(y)f(y)\right\|\\&=\left\|\sum_{j=1}^kh_{i_j}(y)f_{x_{i_j}}(y)-\sum_{j=1}^kh_{i_j}(y)f(y)\right\|\\&\le \sum_{j=1}^kh_{i_j}(y)\left\|f_{x_{i_j}}(y)-f(y)\right\|\stackrel{\eqref{eq:newineq}}{\le} \sum_{i=1}^nh_i(y)\epsilon=\epsilon.
\end{align*}
Therefore, it follows that
$\|a-f\|<\epsilon.$
Since $\epsilon$ is arbitrary, we deduce that $f\in Q$.  This finishes the proof.
\end{proof}
An equivalent description also exists for ideals $I$ inside $C(X,\mathcal{A})$.
\begin{lemma}\cite[Theorem 1 in Ch.V.26.2]{naimark2012normed}\label{lem: Naimark}
Let $X$ be a compact Hausdorff space and $\mathcal{A}$ be any $C^*$-algebra. Every closed ideal $I$ of $C(X, \mathcal{A})$ is of the form $\{f\in C(X, \mathcal{A})\mid f(x)\in I_x, \forall~x\in X\}$, where $I_x$ is the closed ideal of $A$ defined by $I_x=\{f(x)\in A\mid f\in I\}$ for all $x\in X$.
\end{lemma}
Let $S(\mathcal{A})$ denote the state space of $\mathcal{A}$. We equip $S(\mathcal{A})$ with the weak$^*$-topology. Given a subalgebra of $\mathcal{B}\subset C(X,\mathcal{A})$, we can find an equivalence relation on $X\times S(\mathcal{A})$ with which $\mathcal{B}$ can be located (not as a $C^*$-algebra!). Such an identification is handy for us, especially when $\mathcal{A}$ is abelian.
\begin{proposition}\label{prop: locating subalgebras in C(X, A)}
Let $X$ be a compact Hausdorff space and $\mathcal{B}\subset C(X,\mathcal{A})$ be a $C^*$-algebra. Then there is a closed equivalence relation $\sim$ on $X\times S(\mathcal{A})$ for which there exists a continuous injective map $\phi: \mathcal{B}\to C(X\times S(\mathcal{A})/{\sim})$. Moreover, we also have a map $J:X\times S(\mathcal{A})\to S(\mathcal{B})$ which is dual to $\phi$ in the sense that $\phi(b)([x,\tau])=J(x,\tau)(b)$ for all $b\in\mathcal{B}$ and all $(x,\tau)\in X\times S(\mathcal{A})$.
\end{proposition}
\begin{proof}
Given any $(x,\tau), (y,\rho)\in X\times S(\mathcal{A})$, we
define $(x,\tau)\sim (y,\rho)$ iff $\tau(a(x))=\rho(a(y))$ for all $a\in \mathcal{B}$. Viewing $(x,\tau)\in X\times S(\mathcal{A})$ as $\delta_x\otimes\tau\in S(C(X,\mathcal{A}))=S(C(X)\otimes_{min}\mathcal{A})$, this equivalence relation just corresponds to having $(x,\tau)\sim (y,\rho)$ if and only if the two states agree on $\mathcal{B}$. From here, it is clear that this equivalence relation is closed. Moreover, $X\times S(\mathcal{A})/{\sim}$ can be canonically viewed as a subset of $S(\mathcal{B})$ as the image of $X\times S(\mathcal{A})$ under the restriction map $J: S(C(X,\mathcal{A}))\to S(\mathcal{B})$. It is also clear that, viewing $X\times S(\mathcal{A})/{\sim}$ as a subset of $\mathcal{B}$, the composition of the canonical maps $$\phi: \mathcal{B}\to C(S(\mathcal{B})\to C\left(X\times S(\mathcal{A})/{\sim}\right)$$ is dual to the map $J$ in the sense mentioned above. The injectivity of this map is likewise straightforward to verify.
\end{proof}
\begin{remark}
  The map $\phi$ constructed in the above proof may neither be surjective nor an algebraic homomorphism.
\end{remark}
If we restrict ourselves to abelian $C^*$-algebras equipped with $G$-actions, it is well-known that every invariant subalgebra comes from an equivalence relation on $X$. We decide to include the proof for completeness.
\begin{cor}\label{cor: subalgebras in C(X) comes from an equivalence relation}
Let $G\curvearrowright X$ be continuous action on a compact Hausdorff space $X$ and $\mathcal{B}\subset C(X)$ be a unital $G$-invariant $C^*$-subalgebra. Then there exists a $G$-invariant closed equivalence relation $\sim$ on $X$ such that $\mathcal{B}=C(X/{\sim})$.
\end{cor}
\begin{proof}

Let $\mathcal{A}=\mathbb{C}$. Then $S(\mathcal{A})=\{Id\}$, where $Id: \mathbb{C}\rightarrow \mathbb{C}$ denotes the identity map. Hence we may simply identify $X\times S(\mathcal{A})$ with $X$. Then it is not hard to see that the map $\phi: \mathcal{B}\rightarrow C(X/{\sim})$ constructed in Proposition \ref{prop: locating subalgebras in C(X, A)} is an injective $*$-homomorphism. In other words, we can identify $\mathcal{B}$ as a subalgebra in $C(X/{\sim})$ via $\phi$. It suffices to show $\phi$ is surjective.
This is clear since the definition of $\sim$ shows $\mathcal{B}$ separates points in $X/{\sim}$. Therefore, the Stone-Weierstrass theorem shows that $\mathcal{B}=C(X/{\sim})$.

Clearly $\mathcal{B}$ being a $G$-invariant entails that $\sim$ is $G$-invariant, i.e., $x\sim y$ iff $sx\sim sy$ for all $x, y\in X$ and all $s\in G$.
\end{proof}
Given an intermediate subalgebra $Q$, we now formulate an equivalent condition for the splitting of $Q$ when the ambient tensor product is of the form $C(X)\otimes_{\text{min}} C(Y)$.
\begin{proposition}\label{prop: equivalence relation formulation on splitting}
Let $G\curvearrowright X$ and $G\curvearrowright Y$ be two continuous actions on compact Hausdorff spaces $X$ and $Y$. Let $Q$ be an intermediate $G$-$C^*$-subalgebra with $C(Y)\subset Q\subset C(X)\otimes_{\text{min}} C(Y)$.

Then the following are equivalent:

(1) $Q$ is of the form $Q=C(Z)\otimes_{\text{min}} C(Y)$ for some factor $G\curvearrowright Z$ of the action $G\curvearrowright X$.

(2) Given any $x, x'\in X$, if there exists some $y_0\in Y$ such that $q(x, y_0)=q(x', y_0)$ for all $q\in Q$, then $q(x, y)=q(x', y)$ for all $y\in Y$ and all $q\in Q$.
\end{proposition}

\begin{proof}
$(1)\Rightarrow (2)$: Take any $f\in C(Z)$ and set $q:=f\otimes 1_Y\in Q$. From $q(x, y_0)=q(x', y_0)$, we deduce that $f(x)=f(x')$ for all $f\in C(Z)$. Hence, $q(x, y)=q(x', y)$ for all $q\in Q$ by approximating $q$ using algebraic tensor products.

$(2)\Rightarrow (1)$:  Let $\sim_Q$ be the closed equivalence relation on $X\times Y$ defined by $Q$, i.e., $(x, y)\sim_Q(x', y')$ iff $q(x, y)=q(x', y')$ for all $q\in Q$. Note that we may identify $Q$ with $C(\frac{X\times Y}{\sim_Q})$ naturally by Corollary \ref{cor: subalgebras in C(X) comes from an equivalence relation}.

To show that (1) holds, it suffices to observe that we have the following two well-defined maps whose composition is the identity:
\begin{align*}
X/{\sim'}\times Y\rightarrow \frac{X\times Y}{\sim_Q}\rightarrow X/{\sim}\times Y.\\
([x]_{\sim'}, y)\mapsto [(x, y)]_Q\mapsto ([x]_{\sim}, y).
\end{align*}
Here, we define $x\sim x'$ iff there exists some $y\in Y$ such that $q(x, y)=q(x', y)$ for all $q\in Q$; similarly, we define $x\sim' x'$ iff $q(x, y)=q(x',y)$ for all $y\in Y$ and all $q\in Q$. Note that to check the second above map is well-defined, we need to use the assumption $C(Y)\subset Q$ to get that whenever $[(x, y)]_Q=[(x', y')]_Q$, then we have $y=y'$.

Note that condition (2) tells us that $\sim=\sim'$ and hence the composition of the above is the identity.
Finally, set $Z=X/{\sim}=X/{\sim'}$. Then $Q=C(Z)\otimes_{\text{min}} C(X)$.
\end{proof}
The following is an easy observation which we record as a corollary.
\begin{cor}
\label{cor:splitting}
    Let $G\curvearrowright Y$ and $G\curvearrowright X$ be continuous actions on compact Hausdorff spaces. Assume that $ H\leq G$ is a subgroup such that the restriction of the action $G\curvearrowright Y$ to $H$ is minimal and $H\curvearrowright X$ is trivial. Then, any intermediate $G$-$C^*$-subalgebra between $C(Y)$ and $C(X)\otimes_{\text{min}} C(Y)$ splits as a tensor product.
\end{cor}
\begin{proof}
By Proposition~\ref{prop:
 equivalence relation formulation on splitting}, we need to check that for any $x, x'\in X$, if for some $y_0\in Y$, we have $q(x, y_0)=q(x', y_0)$ for all $q\in Q$, then $q(x, y)=q(x', y)$ for all $y\in Y$. By our assumption, take $s_n\in H$ such that $s_ny_0\rightarrow y$. Note that $s_n^{-1}q\in Q$ as $Q$ is $G$-invariant. Then $q(x, y)=\lim_nq(s_nx, s_ny_0)=\lim_n(s_n^{-1}q)(x, y_0)=\lim_n(s_n^{-1}q)(x', y_0)=\lim_nq(s_nx', s_ny_0)=q(x', y)$ for all $q\in Q$.
\end{proof}
Given an invariant $C^*$-subalgebras $Q$ of the form
\[C(X)\subset Q\subset C(X)\otimes_{\text{min}}\mathcal{A},\]
the question of whether $Q$ splits as a tensor product can be translated to an IFT problem by looking at $S(\mathcal{A})$. The following theorem allows us to bridge the commutative and the non-commutative side.
\begin{thm}
\label{thm:bridgecnc}
Let $\mathcal{A}$ be a unital $G$-$C^*$-algebra and $Y$, a $G$-space. 
Let $Q$ be an intermediate $G$-$C^*$-subalgebra of the form
\[C(Y)\subset Q\subset C(Y)\otimes_{\text{min}}\mathcal{A}.\] 
Assume that every intermediate $G$-factor $Z$ with 
\[Y\times S(\mathcal{A})\to Z\to Y\]
is of the form $Y\times K$ for some $G$-factor $K$ of $S(\mathcal{A})$. Then, $Q$ is a tensor product. 
\end{thm}
\begin{proof}
Let us identify $C(Y)\otimes_{\text{min}}\mathcal{A}$ with $C(Y,\mathcal{A})$.
Now, given $G$-invariant intermediate $G$-$C^*$-algebra $Q$ with $C(Y)\subset Q\subset C(Y)\otimes_{\text{min}} \mathcal{A}$, we construct an intermediate factor. Let $J: Y\times S(\mathcal{A})\to S(Q)$ be the restriction map defined by $$J(y,\tau)(q):=\tau(q(y)).$$ Note that the restriction of states gives us a composition map
$\pi: Y\times S(\mathcal{A})\overset{J}{\rightarrow} S(Q)\overset{\text{Res}}{\rightarrow} S(C(Y))$. We claim that $$J(Y\times S(\mathcal{A}))=\left\{\delta_y\times\varphi|_{Q_y}:y\in Y, \varphi\in S(\mathcal{A})\right\}.$$ 
If $\tau\in S(Q)$ is such that $\tau\in J(Y\times S(\mathcal{A}))$, then $\tau=J(y,\eta)$ for some $\eta\in S(\mathcal{A})$ and $y\in Y$. Then, 
\[\tau(q)=J(y,\eta)(q)=\eta(q(y))=\eta|_{Q_y}(q(y))=\delta_y\times\eta|_{Q_y}(q),~q\in Q.\]
Hence $\tau=\delta_y\times \eta|_{Q_y}$.
Then it is clear that $$\text{Res}\left(J(Y\times S(\mathcal{A}))\right)=\left\{\delta_y: y\in Y\right\}.$$
Consequently, it follows that $\pi$ is a map onto $Y$, i.e.,
\[\pi: Y\times S(\mathcal{A})\overset{J}{\rightarrow} J(Y\times S(\mathcal{A}))\overset{\text{Res}}{\rightarrow}Y.\]
From our assumption, we know that $J(Y\times S(\mathcal{A}))=Y\times K$ for some $G$-factor $K$ of $S(\mathcal{A})$.
We now claim that $\text{Res}^{-1}(y)=\delta_y\times S(Q_y)$ for all $y\in Y$. Indeed, if $\delta_{y'}\times\varphi|_{Q_{y'}}\in \text{Res}^{-1}(y)$, then we see that
\begin{align*}
&\delta_{y'}\times\varphi|_{Q_{y'}}(f)=f(y)~\forall f\in C(Y)\\&\iff \varphi|_{Q_{y'}}(f(y'))=f(y)~\forall f\in C(Y)\\&\iff f(y')=f(y)~\forall f\in C(Y)\iff y'=y.    
\end{align*}
Since $S(Q_y)=S(\mathcal{A})|_{Q_y}$ (one side of the inclusion follows from the Hahn-Banach extension theorem, and the other side is evident from the definition of the $\text{Res}$ map), the claim follows. 
Therefore, $S(Q_y)=K$ for all $y\in Y$. 
We now show that $Q_y=Q_{y'}$ for all $y,y'\in Y$ and then apply Lemma \ref{lem:smallparts} to finish the proof. 

Write $K=\frac{S(\mathcal{A})}{\sim}$ for some $G$-invariant closed equivalence relation on $S(\mathcal{A})$. Fix any $y\in Y$. Note that $S(Q_y)=K=\frac{S(\mathcal{A})}{\sim}$ means for any $\phi\in S(\mathcal{A})$, we have $\phi|_{Q_y}=[\phi]$, the equivalence class of $\phi$ under $\sim$.
Then we check that the following claim holds true:

Claim: For any $\phi,\psi\in S(\mathcal{A})$, $\phi\sim\psi$ $\Leftrightarrow$ $\phi(q)=\psi(q)$ for all $q\in Q_y$. 

For the $\Rightarrow$ direction, we have $[\phi]=[\psi]:=[\Phi]$, thus $\phi|_{Q_y}=\Phi|_{Q_y}=\psi|_{Q_y}$, hence $\phi(q)=\psi(q)$ for all $q\in Q_y$. For the $\Leftarrow$ direction, note that $\phi|_{Q_y}=\psi|_{Q_y}$, thus $[\phi]=\phi|_{Q_y}=\psi|_{Q_y}=[\psi]$, hence $\phi\sim \psi$.

Finally, let us show $Q_y=Q_{y'}$ for any $y,y'\in Y$. Assuming the equality does not hold, then without loss of generality, we may assume there is some $a\in Q_y\setminus Q_{y'}$. Using the Hahn-Banach extension theorem, we can find a bounded linear functional $\omega$ on $\mathcal{A}$ such that $\omega|_{Q_{y'}}=0$ and $\omega(a)\ne 0$. Write $\omega=c_1\omega_1-c_2\omega_2+ic_3\omega_3-ic_4\omega_4$, where $\omega_i\in S(\mathcal{A})$ and $c_i\in\mathbb{R}$ for each $i=1,2,3,4$. Since $\omega(q)=0$ for all $q\in Q_{y'}$, it follows that $c_1=c_2$, $c_3=c_4$ and also $\omega_1(q)=\omega_2(q)$ (if $c_1, c_2\neq 0$), $\omega_3(q)=\omega_4(q)$ (if $c_3, c_4\neq 0$) for all $q\in Q_{y'}$. This shows that $\omega_1\sim\omega_2$ (if $c_1,c_2\ne0$) and $\omega_3\sim \omega_4$ (if $c_3,c_4\ne0$) by the above claim for $y'$. Once again, the above claim (for $y$) shows that $\omega_1(q)=\omega_2(q)$ (if $c_1,c_2\ne0$), $\omega_3(q)=\omega_4(q)$ (if $c_3,c_4\ne0$)  for all $a\in Q_y$ and thus $\omega(q)=0$ for all $q\in Q_y$. But this contradicts with $\omega(a)\neq 0$.
\end{proof}
\subsection{Strongly proximal, boundary Actions}
Consider a $G$-space $X$, i.e., $X$ is a compact Hausdorff space over which $G$ acts via homeomorphisms.
The action $G \curvearrowright X$ is called a strongly proximal action if for any $\nu \in \text{Prob}(X)$, $$\left\{\delta_x: x\in X\right\}\cap\overline{G\nu}^{\text{weak}^*}\neq \emptyset.$$ If $G\curvearrowright X$ is a minimal strongly proximal action, then we say it is a boundary action or $X$ is a $G$-boundary.
Clearly, $X$ is a $G$-boundary if for any $\nu\in \text{Prob}(X)$,
$$\left\{\delta_x: x\in X\right\}\subset\overline{G\nu}^{\text{weak}^*}.$$
By $\partial_FG$, we denote the Furstenberg boundary of the group $G$. It is the maximal $G$-boundary, in that any other $G$-boundary $Y$ is a $G$-factor of $\partial_FG$. Very recently, this concept has been instrumental in analyzing properties of groups. For instance, Kalantar and Kennedy \cite{kalantar_kennedy_boundaries} provided a dynamical description of $C^*$-simplicity through the interaction of $G\curvearrowright\partial_FG$ on the Furstenberg boundary $\partial_FG$.

We exploit the boundary actions to give various examples of a tensor product splitting theorem. We end this discussion with the following class of groups for which an intermediate splitting theorem holds. This also illustrates our point of view of pairing the boundary action with that of a minimal action.
\begin{example}
Let $G=\Gamma_1\times \Gamma_2$. Assume that $\Gamma_2$ is amenable and $\Gamma_1$ is nonamenable with a trivial amenable radical. Let $X$ be a minimal $\Gamma_2$-space. Denote by $\partial_FG$ the Furstenberg boundary associated with $G$. By \cite[Corollary 8]{furman_kernel}, we see that $\Gamma_2$ is contained inside $\text{Ker}(G\curvearrowright\partial_FG)$. Since any $G$-boundary $\partial G$ is a $G$-factor of $\partial_FG$, it follows that $\Gamma_2$ acts trivially on $\partial G$. It now is a consequence of Corollary~\ref{cor:splitting} that every intermediate $G$-$C^*$-subalgebra $\mathcal{C}$ of the form
\[C(X)\subset\mathcal{C}\subset C(X)\otimes_{\text{min}} C(\partial G)\]
splits.
\end{example}
\subsection{Property (S)}
\label{subsec:propertyS}
We also briefly recall the notion of Property~(S) before proceeding further. Let $\mathcal{A}$ and $\mathcal{B}$ be unital $C^*$-algebras. For a state $\psi$ on $\mathcal{B}$, we can associate the conditional expectation $L_{\psi} : \mathcal{A}\otimes_{\text{min}}\mathcal{B}\to \mathcal{A}$ defined by $a\otimes b\mapsto a\psi(b)$. Similarly, for a state $\varphi$ on $\mathcal{A}$, we can associate a conditional expectation $R_{\varphi}:\mathcal{A}\otimes_{\text{min}}\mathcal{B}\to \mathcal{B}$ defined by $a\otimes b\mapsto \varphi(a)b$.  Given unital $C^*$-subalgebras $\tilde{\mathcal{A}}\subset\mathcal{A}$ and  $\tilde{\mathcal{B}}\subset\mathcal{B}$, the set
\begin{align*}&F\left(\tilde{\mathcal{A}}, \tilde{\mathcal{B}}, \mathcal{A}\otimes_{\text{min}}\mathcal{B}\right)\\&:=\left\{x\in\mathcal{A}\otimes_{\text{min}}\mathcal{B}: R_{\varphi}(x)\in \tilde{\mathcal{B}}, L_{\psi}(x)\in\tilde{\mathcal{A}}
~\forall \varphi\in S(\mathcal{A}), \psi \in S(\mathcal{B})\right\}\end{align*}
is called the Fubini product of $\tilde{\mathcal{A}}$ and $\tilde{\mathcal{B}}$.  $\mathcal{A}$ is said to have property~(S) (see~\cite{Wassermann}) if $F(\mathcal{A}, \tilde{\mathcal{B}}, \mathcal{A}\otimes_{\text{min}}\mathcal{B}) = \mathcal{A}\otimes_{\text{min}}\tilde{\mathcal{B}}$ for each $C^*$-subalgebra $\tilde{\mathcal{B}}\subset \mathcal{B}$. For example, it is known that every nuclear $C^*$-algebra has property~(S). 
\section{Simple Invariant \texorpdfstring{$C^*$}{}-algebras}
\label{sec:simpleinvariantC*-algebras}
In this section, we prove an intermediate splitting theorem under some additional assumptions on the group actions. Our first goal is to follow the techniques of Zacharias~\cite{zacharias2001splitting} to obtain Proposition~\ref{prop:intromainsplittingC*}. This uses essentially the same methods as in the original, but we provide the details for the sake of completeness. Before giving the proof, we need two technical results (which are again minor modifications of the arguments found in \cite{zacharias2001splitting}). The first says that in a $G$-simple $C^*$-algebra, the unit can be written as an image of an elementary u.c.p map. We do a straightforward modification of this well-known result for simple $C^*$-algebras (see, for example, \cites{cuntz1977structure} or \cite[Exercise~4.8]{Rordambook}).
\begin{proposition}
\label{prop:elementaryapprox}
Let $\mathcal{A}$ be  a $G$-simple unital $C^*$-algebra. Let $a\in\mathcal{A}$ be a non-zero positive element with $\|a\|=1$. Then there exist $\tilde{b}_1, \tilde{b}_2,\ldots,\tilde{b}_m\in\mathcal{A}$ and $t_1, t_2, \ldots,t_m\in G$ such that
\begin{equation}\label{forone}1_{\mathcal{A}}=\sum_{i=1}^m\tilde{b}_i\sigma_{t_i}(a)\tilde{b}_i^*.\end{equation}
We may further assume that
\[\sum_{i=1}^m\tilde{b}_i\tilde{b}_i^*\le 2.\]
\begin{proof}

Given $a\in\mathcal{A}^{+}$, consider $$I_{\text{alg}}^a=\text{Span}\left\{b\sigma_{s}(a)c:~c,b\in\mathcal{A}, s\in G\right\}.$$
Since $\overline{I_{\text{alg}}^a}$ is a closed $G$-invariant ideal in $\mathcal{A}$, it follows that $1_{\mathcal{A}}\in \overline{I_{\text{alg}}^a}$. In particular, given $\epsilon<1$, we can find an element of the form $\sum_{i=1}^nb_i\sigma_{s_i}(a)c_i\in I_{\text{alg}}^a$ such that
\[\left\|1_{\mathcal{A}}-\sum_{i=1}^nb_i\sigma_{s_i}(a)c_i\right\|<\epsilon<1.\]
In particular, this implies that $\tilde{a}=\sum_{i=1}^nb_i\sigma_{s_i}(a)c_i$ is an invertible element. Therefore,
\begin{equation*}
1_{\mathcal{A}}=\tilde{a}^{-1}\tilde{a}=\sum_{i=1}^n\tilde{a}^{-1}b_i\sigma_{s_i}(a)c_i.
\end{equation*}
Rewriting $\tilde{a}^{-1}b_i=a_i$, we obtain that
\[1_{\mathcal{A}}=\tilde{a}^{-1}\tilde{a}=\sum_{i=1}^na_i\sigma_{s_i}(a)c_i.\]
Therefore,\[\sum_{i=1}^na_i\sigma_{s_i}(a)c_i+\sum_{i=1}^nc_i^*\sigma_{s_i}(a)a_i^*\ge 1_{\mathcal{A}}.\]
On the other hand,
\begin{align*}
&\sum_{i=1}^na_i\sigma_{s_i}(a)a_i^*+\sum_{i=1}^nc_i^*\sigma_{s_i}(a)c_i-\left(\sum_{i=1}^na_i\sigma_{s_i}(a)c_i+\sum_{i=1}^nc_i^*\sigma_{s_i}(a)a_i^*\right)\\&=\sum_{i=1}^na_i\sigma_{s_i}(a)(a_i^*-c_i)+\sum_{i=1}^nc_i^*\sigma_{s_i}(a)(c_i-a_i^*)\\&=\sum_{i=1}^na_i\sigma_{s_i}(a)(a_i^*-c_i)-\sum_{i=1}^nc_i^*\sigma_{s_i}(a)(a_i^*-c_i)\\&=\sum_{i=1}^n(a_i-c_i^*)\sigma_{s_i}(a)(a_i^*-c_i)\ge 0.
\end{align*}
Consequently, it follows that
\[\sum_{i=1}^na_i\sigma_{s_i}(a)a_i^*+\sum_{i=1}^nc_i^*\sigma_{s_i}(a)c_i\ge 1_{\mathcal{A}}.\]
Let us denote $\sum_{i=1}^na_i\sigma_{s_i}(a)a_i^*+\sum_{i=1}^nc_i^*\sigma_{s_i}(a)c_i=\tilde{b}$ for the sake of simplicity. Since $\tilde{b}\ge 1_{\mathcal{A}}$, we see that the spectrum $\text{Sp}(\tilde{b}-1_{\mathcal{A}})\subset [0,\infty)$. Therefore, $\text{Sp}(\tilde{b})\subset [1,\infty)$. In particular, $\tilde{b}$ is invertible. Moreover, $\tilde{b}^{\frac{-1}{2}}$ is well-defined and is a self-adjoint element. Therefore,
\begin{align*}
 1_{\mathcal{A}}&=\tilde{b}^{\frac{-1}{2}}\tilde{b}\tilde{b}^{\frac{-1}{2}}\\&= \tilde{b}^{\frac{-1}{2}}\left(\sum_{i=1}^na_i\sigma_{s_i}(a)a_i^*+\sum_{i=1}^nc_i^*\sigma_{s_i}(a)c_i\right)\tilde{b}^{\frac{-1}{2}}\\&=\sum_{i=1}^n\tilde{b}^{\frac{-1}{2}}a_i\sigma_{s_i}(a)a_i^*\tilde{b}^{\frac{-1}{2}}+\sum_{i=1}^n\tilde{b}^{\frac{-1}{2}}c_i^*\sigma_{s_i}(a)c_i\tilde{b}^{\frac{-1}{2}}.
\end{align*}
Let us now take $m=2n$, and rename $\tilde{b}^{\frac{-1}{2}}a_i=\tilde{b}_i$, $\tilde{b}^{\frac{-1}{2}}c_i^*=\tilde{b}_{n+i}$, $t_i=s_i$  and $t_{i+n}=s_i$ for $i=1,2,\ldots,n$. The first conclusion follows.

Now, given any $a\in\mathcal{A}^+$, we denote by $[a]_{\epsilon}$, the spectral projection onto the subspace in the universal representation, where $a\ge \epsilon$. Let $f_{\epsilon} : [0, 1]\to[0, 1]$ be
defined by $f_{\epsilon}(t) = t$ for $t\in(\epsilon, 1]$, $f_{\epsilon}(t) = 2(t-\frac{\epsilon}{2})$ for $t\in (\frac{\epsilon}{2}, \epsilon]$ and $f_{\epsilon}(t)=0$ otherwise. Let $c= f_{1}(a)af_{1}(a)\in\mathcal{A}^+$. Now, arguing as above, we can find $y_1, y_2,\ldots,y_m\in\mathcal{A}$ and $t_1, t_2, \ldots,t_m\in G$ such that
\[\sum_{i=1}^my_i\sigma_{t_i}(c)y_i^*=1_\mathcal{A}.\]
Let $\tilde{a}_i=\sigma_{t_i^{-1}}(y_i)f_1(a)$. Then, $$\tilde{a}_i[a]_{\frac{1}{2}}=\sigma_{t_i^{-1}}(y_i)f_1(a)[a]_{\frac{1}{2}}=\sigma_{t_i^{-1}}(y_i)f_1(a)=\tilde{a}_i, \text{ and }  2[a]_{\frac{1}{2}}a \ge [a]_{\frac{1}{2}}.$$
Therefore, we see that
\begin{align*}
1_{\mathcal{A}}&=\sum_{i=1}^my_i\sigma_{t_i}(f_1(a))\sigma_{t_i}(a)\sigma_{t_i}(f_1(a))y_i^*\\&=\sum_{i=1}^m\sigma_{t_i}(\tilde{a}_i)\sigma_{t_i}(a)\sigma_{t_i}(\tilde{a}_i^*)\\&=\sum_{i=1}^m\sigma_{t_i}\left(\tilde{a}_ia\tilde{a}_i^*\right)\\&=\sum_{i=1}^m\sigma_{t_i}\left(\tilde{a}_i[a]_{\frac{1}{2}}a\tilde{a}_i^*\right)\ge \frac{1}{2}\sum_{i=1}^m\sigma_{t_i}\left(\tilde{a}_i[a]_{\frac{1}{2}}\tilde{a}_i^*\right)\\&=\frac{1}{2}\sum_{i=1}^m\sigma_{t_i}\left(\tilde{a}_i\right)\sigma_{t_i}\left(\tilde{a}_i^*\right)
\end{align*}
Letting $\tilde{b}_i=\sigma_{t_i}(\tilde{a}_i)$, we see that
\[\sum_{j=1}^m\tilde{b}_i\sigma_{t_i}(a)\tilde{b}_i^*=1_\mathcal{A}\text{ and }\sum_{i=1}^m\tilde{b}_i\tilde{b}_i^*\le 2.\]
\end{proof}
\end{proposition}
We will use the elementary u.c.p maps to approximate states on $\mathcal{A}$. The proof is the same as in \cite[Proposition~2.4]{zacharias2001splitting}. We include it for the sake of completion.
\begin{proposition}
\label{prop:approxucpstate}
Let $\mathcal{A}$ be a $G$-simple unital $C^*$-algebra. Let $\varphi$ be a state on $\mathcal{A}$. Then there exists a net of ucp maps $\Psi_{\lambda}:\mathcal{A}\to\mathcal{A}$ of the form
\[\Psi_{\lambda}(\cdot)=\sum_{j}\tilde{a}_j(\lambda)\sigma_{t_j}(\cdot)\tilde{a}_j(\lambda)^*,\] where $\tilde{a_j}(\lambda)\in\mathcal{A}$ and $t_j\in G$, such that
\[\Psi_{\lambda}(a)\xrightarrow{\lambda\to\infty}\varphi(a)1_{\mathcal{A}},~\forall~a\in\mathcal{A}.\]
\begin{proof}
Since the convex combination of pure states is dense in $S(\mathcal{A})$ in the weak$^*$-topology, it is enough to prove the claim for a pure state $\varphi$. Towards this end, let $\varphi$ be a pure state on $\mathcal{A}$. Then, using \cite[Lemma~2.1]{zacharias2001splitting}, we can find a net $\{a_{\lambda}\}\subset \mathcal{A}^+$ (which is a net of approximate identity for a specific subalgebra of $\mathcal{A}$) such that
\[\left\|a_{\lambda}xa_{\lambda}-\varphi(x)a_{\lambda}^2\right\|\xrightarrow{}0,~\forall x\in\mathcal{A}.\] Using Proposition~\ref{prop:elementaryapprox}, we can find $\tilde{a}_{1(\lambda)}, \tilde{a}_{2(\lambda)},\ldots,\tilde{a}_{m(\lambda)}\in\mathcal{A}$ and $t_{1(\lambda)}, t_{2(\lambda)}, \ldots,t_{m(\lambda)}\in G$ such that
\begin{equation*}\label{newforone}1_{\mathcal{A}}=\sum_{i=1}^{m(\lambda)}\tilde{a}_{i(\lambda)}\sigma_{t_{i(\lambda)}}(a_{\lambda}^2)\tilde{a}_{i(\lambda)}^*.\end{equation*}
Let $\Psi_{\lambda}:\mathcal{A}\to\mathcal{A}$  be defined by
\[\Psi_{\lambda}(x)=\sum_{i=1}^{m(\lambda)}\tilde{a}_{i(\lambda)}\sigma_{t_{i(\lambda)}}(a_{\lambda}xa_{\lambda})\tilde{a}_{i(\lambda)}^*\]
Consequently, we see that
\begin{equation}
\label{eq:ucp}\Psi_{\lambda}(1_{\mathcal{A}})=\sum_{i=1}^{m(\lambda)}\tilde{a}_{i(\lambda)}\sigma_{t_{i(\lambda)}}(a_{\lambda}^2)\tilde{a}_{i(\lambda)}^*=1_{\mathcal{A}}.\end{equation}
We now observe that for each $x\in\mathcal{A}$,
\begin{align*}
&\left\|\Psi_{\lambda}(x)-\varphi(x)1_{\mathcal{A}}\right\|\\
&=\left\|\sum_{i=1}^{m(\lambda)}\tilde{a}_{i(\lambda)}\sigma_{t_{i(\lambda)}}(a_{\lambda}xa_{\lambda})\tilde{a}_{i(\lambda)}^*-\varphi(x)1_{\mathcal{A}}\right\|\\
&=\left\|\sum_{i=1}^{m(\lambda)}\tilde{a}_{i(\lambda)}\sigma_{t_{i(\lambda)}}\left(a_{\lambda}xa_{\lambda}-\varphi(x)a_{\lambda}^2\right)\tilde{a}_{i(\lambda)}^*+\sum_{i=1}^{m(\lambda)}\tilde{a}_{i(\lambda)}\sigma_{t_{i(\lambda)}}(\varphi(x)a_{\lambda}^2)\tilde{a}_{i(\lambda)}^*-\varphi(x)1_{\mathcal{A}}\right\|\\
&\le\left\|\sum_{i=1}^{m(\lambda)}\tilde{a}_{i(\lambda)}\sigma_{t_{i(\lambda)}}\left(a_{\lambda}xa_{\lambda}-\varphi(x)a_{\lambda}^2\right)\tilde{a}_{i(\lambda)}^*\right\|+\left\|\sum_{i=1}^{m(\lambda)}\tilde{a}_{i(\lambda)}\sigma_{t_{i(\lambda)}}(\varphi(x)a_{\lambda}^2)\tilde{a}_{i(\lambda)}^*-\varphi(x)1_{\mathcal{A}}\right\|\\&=\left\|\sum_{i=1}^{m(\lambda)}\tilde{a}_{i(\lambda)}\sigma_{t_{i(\lambda)}}\left(a_{\lambda}xa_{\lambda}-\varphi(x)a_{\lambda}^2\right)\tilde{a}_{i(\lambda)}^*\right\|+\left\|\varphi(x)\left(\sum_{i=1}^{m(\lambda)}\tilde{a}_{i(\lambda)}\sigma_{t_{i(\lambda)}}(a_{\lambda}^2)\tilde{a}_{i(\lambda)}^*\right)-\varphi(x)1_{\mathcal{A}}\right\|\\
&\stackrel{\eqref{eq:ucp}}{=} \left\|\sum_{i=1}^{m(\lambda)}\tilde{a}_{i(\lambda)}\sigma_{t_{i(\lambda)}}\left(a_{\lambda}xa_{\lambda}-\varphi(x)a_{\lambda}^2\right)\tilde{a}_{i(\lambda)}^*\right\|.
\end{align*}
Let $\Phi_{\lambda}:\mathcal{A}\to\mathcal{A}$  be defined by
\[\Phi_{\lambda}(x)=\sum_{i=1}^{m(\lambda)}\tilde{a}_{i(\lambda)}\sigma_{t_{i(\lambda)}}(x)\tilde{a}_{i(\lambda)}^*,\]
we see that
\[\left\|\Psi_{\lambda}(x)-\varphi(x)1_{\mathcal{A}}\right\|\le \left\|\Phi_{\lambda}\left(a_{\lambda}xa_{\lambda}-\varphi(x)a_{\lambda}^2\right)\right\|.\]
Moreover, since,
\[\sum_{j}\tilde{a}_j(\lambda)\tilde{a}_j(\lambda)^*\le 2,\]
it follows that $\|\Phi_{\lambda}\|\le 2$ for each $\lambda$. Consequently, we obtain that
\[\left\|\Psi_{\lambda}(x)-\varphi(x)1_{\mathcal{A}}\right\|\le \left\|\Phi_{\lambda}\left(a_{\lambda}xa_{\lambda}-\varphi(x)a_{\lambda}^2\right)\right\|\le 2 \left\|a_{\lambda}xa_{\lambda}-\varphi(x)a_{\lambda}^2\right\|\xrightarrow[]{\lambda\to\infty}0.\]
\end{proof}
\end{proposition}
We are ready to prove the intermediate splitting Proposition~\ref{prop:intromainsplittingC*}. As usual, we show that the slice map induced by every state on the underlying $C^*$-algebra preserves the intermediate $C^*$-algebra. For any state $\varphi$ on $\mathcal{A}$, the corresponding slice map $\mathbb{E}_{\varphi}:\mathcal{A}\otimes_{\text{min}}\mathcal{B}\to\mathcal{B}$ is given by
\[\mathbb{E}_{\varphi}(a\otimes b)=\varphi(a)b, a\in\mathcal{A},b\in\mathcal{B}.\]
\begin{proof}[Proof of Proposition~\ref{prop:intromainsplittingC*}]
We prove it in a couple of steps.\\
\noindent
First, we show that for every $\varphi\in S(\mathcal{A})$, $\mathbb{E}_{\varphi}(\mathcal{C})\subset\left(\mathcal{C}\cap \left(1\otimes_{\text{min}} \mathcal{B}\right)\right).$

Let $\varphi\in S(\mathcal{A})$. Let $a\in\mathcal{C}$ and $\epsilon>0$ be given. We can find $a_1, a_2,\ldots, a_m\in \mathcal{A}$ and $b_1, b_2,\ldots, b_m\in \mathcal{B}$ such that
\[\left\|a-\sum_{i=1}^ma_i\otimes b_i\right\|<\frac{\epsilon}{3}.\] Since $\mathbb{E}_{\varphi}$ is contractive and $\mathbb{E}_{\varphi}(a_i\otimes b_i)=\varphi(a_i)1_{\mathcal{A}}\otimes b_i$, it follows that
\[\left\|\mathbb{E}_{\varphi}(a)-\sum_{i=1}^m\varphi(a_i)1_{\mathcal{A}}\otimes b_i\right\|<\frac{\epsilon}{3}.\]
Let $M=\max_{1\le i\le m}\|b_i\|$. We denote by $\text{UCP}(\mathcal{A})$, the collection of all unital completely positive maps from $\mathcal{A}$ to $\mathcal{A}$. Using Proposition~\ref{prop:approxucpstate} to $\Gamma\curvearrowright \mathcal{A}$, we can find $$\Phi_{a_{\lambda}}\in \text{UCP}(\mathcal{A})\text{ with } \Phi_{a_{\lambda}}(x)=\sum_{i=1}^n\tilde{a}_{i(\lambda)}\sigma_{t_i(\lambda)}(x)\tilde{a}_{i(\lambda)}^*,~x\in\mathcal{A},~t_i(\lambda)\in \Gamma$$ such that
\[\left\|\Phi_{a_{\lambda}}(a_i)-\varphi(a_i)\right\|<\frac{\epsilon}{3mM},~i=1,2,\ldots,m.\]
We point out that $\Phi_{a_{\lambda}}\otimes \text{id}: \mathcal{A}\otimes_{\text{min}} \mathcal{B}\to \mathcal{A}\otimes_{\text{min}} \mathcal{B}$ is a completely bounded map defined by $\left(\Phi_{a_{\lambda}}\otimes \text{id}\right)(a\otimes b)=\Phi_{a_{\lambda}}(a)\otimes b$ (see \cite[Corollary~3.5.5]{BroOza08}). Therefore, we see that
\begin{align*}
&\left\|\left(\Phi_{a_{\lambda}}\otimes \text{id}\right)\left(\sum_{i=1}^ma_i\otimes b_i\right)-\sum_{i=1}^m\varphi(a_i)1_{\mathcal{A}}\otimes b_i\right\|\\&=\left\|\sum_{i=1}^m\Phi_{a_{\lambda}}(a_i)\otimes b_i-\sum_{i=1}^m\varphi(a_i)1_{\mathcal{A}}\otimes b_i\right\|\\&\le \sum_{i=1}^m\left\|\Phi_{a_{\lambda}}(a_i)-\varphi(a_i)\right\|\|b_i\|\le \sum_{i=1}^m\frac{\epsilon}{3mM}\|b_i\|<\frac{\epsilon}{3}.\end{align*}
Using the triangle inequality, we now see that
\begin{align*}
\left\|\mathbb{E}_{\varphi}(a)-\left(\Phi_{a_{\lambda}}\otimes \text{id}\right)(a)\right\|&\le \left\|\mathbb{E}_{\varphi}(a)-\mathbb{E}_{\varphi}\left(\sum_{i=1}^ma_i\otimes b_i\right)\right\|\\&+\left\|\sum_{i=1}^m\varphi(a_i)1_{\mathcal{A}}\otimes b_i-\left(\Phi_{a_{\lambda}}\otimes \text{id}\right)\left(\sum_{i=1}^ma_i\otimes b_i\right)\right\|\\&+\left\|\left(\Phi_{a_{\lambda}}\otimes \text{id}\right)\left(\sum_{i=1}^ma_i\otimes b_i\right)-\left(\Phi_{a_{\lambda}}\otimes \text{id}\right)(a)\right\|\\&\le \frac{\epsilon}{3}+ \frac{\epsilon}{3}+ \frac{\epsilon}{3}=\epsilon.
\end{align*}
Since $\Phi_{a_{\lambda}}\otimes \text{id}$ leaves $\mathcal{C}$ invariant, the claim follows. On the other hand, since $\left(\mathcal{C}\cap \left(1\otimes_{\text{min}} \mathcal{B}\right)\right)\subset\mathbb{E}_{\varphi}(\mathcal{C})$, it follows that $$ \left(\mathcal{C}\cap \left(1\otimes_{\text{min}} \mathcal{B}\right)\right)=\mathbb{E}_{\varphi}(\mathcal{C}).$$ Since $\mathcal{A}$ has property (S), we can now appeal to \cite[Proposition~10]{Wassermann} to finish the proof.
\end{proof}
\begin{remark}
\label{rem:ideal}
If we drop the assumption of $\Gamma$-simplicity of $\mathcal{A}:=C(X)$, then $\mathcal{C}$ may not split. Indeed, since $\Gamma\curvearrowright X$ is not minimal, some $\Gamma$-invariant closed  nontrivial two-sided ideal $J$ in $C(X)$ exist. It is now clear that $\mathcal{C}:=C(X)\otimes_{\text{min}} \mathbb{C}+ J\otimes_{\text{min}} C(Y)$ is a $\Gamma$-invariant intermediate algebra that does not split.
\end{remark}
Below, we present various examples such that Proposition~\ref{prop:intromainsplittingC*} could be applied.
\begin{example}
Given a group $G$, let $\Gamma\le G$ be a normal abelian subgroup. Let $\mathcal{B}=C_r^*(\Gamma)$. Note that $\Gamma\curvearrowright\mathcal{B}$ is trivial with respect to the conjugation action. Moreover, $G\curvearrowright C_r^*(\Gamma)$ via the conjugation action. Let $\mathcal{A}$ be a $\Gamma$-simple unital $C^*$-algebra with property~(S). Note that any $G$-invariant closed two-sided ideal is also $\Gamma$-invariant, hence trivial. So, $\mathcal{A}$ is also $G$-simple. Consider the inclusion of $G$-$C^*$-algebras
\[\mathcal{A}\otimes_{\text{min}} 1\subset\mathcal{C}\subset \mathcal{A}\otimes_{\text{min}} C_r^*(\Gamma)\]
Then, it follows from Proposition~\ref{prop:intromainsplittingC*} that $\mathcal{C}$ splits.
\end{example}
We briefly recall the notion of acylindrically hyperbolic groups and refer the reader to \cite{OS} for more details. Consider an isometric action $G\curvearrowright(X,d)$ on a metric space. This action is called acylindrical if, given any $\epsilon > 0$, there exists $\delta, N > 0$ such that for any two points $x,y$ in $X$ with $d(x,y) \ge \delta$, we have
\[\left|\{g \in G: d(x, gx) \le \epsilon \text{ and } d(y, gy) \le \epsilon\}\right| \le N.\]
A group $G$ is called acylindrically hyperbolic if it can admit a non-elementary acylindrical action on a hyperbolic space $S$. Here, the action is non-elementary means the limit set of $G$ on the Gromov boundary $\partial S$ contains at least three points. The notion of the plump subgroup was introduced in \cite{amrutam2021} and was used to determine the structure of the intermediate subalgebras of the crossed product. 
\begin{proposition}
\label{prop:acylindplump}
Let $N$ be a normal subgroup of an acylindrically hyperbolic group $G$. Assume that $G$ is $C^*$-simple. Then, $N$ is plump in $G$ in the sense of \cite[Definition 1.1]{ursu2019relative}.
\begin{proof}

It is enough to show that $C_{G}(N)$, the centralizer of $N$ inside $G$, is $\{e\}$ and then the claim follows from \cite[Theorem~1.3]{ursu2019relative}. Since $G$ is acylindrically hyperbolic, it admits a non-elementary acylindrical action on a hyperbolic space $S$. We observe that $N$ must be infinite. Using \cite[Theorem 3.7]{AO}, we see that every infinite normal subgroup is $s$-normal.   Hence, the action $N \curvearrowright S$ is also non-elementary by \cite[Lemma 7.1]{OS}. Also, by
\cite[Theorem 1.2]{OS}, $N$ contains at least one loxodromic element, say $n$. Hence, by \cite[Corollary 6.9]{OS}, $C_{G}(n)$ is virtually cyclic, hence amenable. Therefore, $C_{G}(N)$ is a normal amenable subgroup of $G$. Since $G$ has trivial amenable radical, $C_{G}(N) = \{e\}$. This concludes
the proof.
\end{proof}
\end{proposition}
\begin{cor}
\label{cor:splittingacylindricallyhyperbolic}
Let $G$ be a $C^*$-simple acylindrically hyperbolic group with property-AP. Let $\mathcal{B}$ be a unital $G$-$C^*$-algebra such that $G\curvearrowright\mathcal{B}$ is non-faithful. Then, every intermediate $G$-$C^*$-subalgebra $\mathcal{C}$ with
\[C_r^*(G)\otimes_{\text{min}} \mathbb{C}\subset\mathcal{C}\subset C_r^*(G)\otimes_{\text{min}}\mathcal{B}\]
splits as a tensor product.
\end{cor}
\begin{proof} Let $\Gamma=\text{Ker}(G\curvearrowright\mathcal{B})$. Then, $\Gamma\triangleleft G$ is a normal subgroup. Using Proposition~\ref{prop:acylindplump}, we see that $\Gamma$ is plump in $G$.  In particular, for every non-zero positive element in $C_r^*(G)$, we can find elements from $\Gamma$ such that after averaging by these elements, we can get arbitrarily close to a non-zero constant. Therefore, $C_r^*(G)$ is $\Gamma$-simple. Consider the inclusion
\[C_r^*(G)\otimes_{\text{min}} 1\subset\mathcal{C}\subset C_r^*(G)\otimes_{\text{min}}\mathcal{B}\]
Since $G$ has property AP, it follows from \cite[Propsition~2.4]{suzuki2017group} that $C_r^*(G)$ has property (S). It now follows from Proposition~\ref{prop:intromainsplittingC*} that $\mathcal{C}$ splits.
\end{proof}
Under additional assumptions on the group, we have more classes of examples. Recall that the amenable radical of a group $G$ is the largest amenable normal subgroup inside the group $G$. We denote it by $R_a(G)$.
\begin{cor}
Let $G$ be a non-amenable group with a nontrivial amenable radical $R_a(G)$. Let $\partial G$ be a $G$-boundary. Let $\mathcal{A}$ be a $H$-simple unital $G$-$C^*$-algebra with property (S), where $H\le R_a(G)$. Then, every intermediate $G$-$C^*$-algebra $\mathcal{C}$ of the form
\[\mathcal{A}\subset\mathcal{C}\subset \mathcal{A}\otimes_{\text{min}} C(\partial G)\]
splits.
\end{cor}
\begin{proof}
It follows from \cite[Proposition 7]{furman_kernel} (see also \cite[Proposition~2.8]{BKKO}) that $R_a(G)=\text{Ker}(G\curvearrowright\partial_FG)$. Since $H\le R_a(G)$, we see that $H\le \text{Ker}(G\curvearrowright\partial_FG)$. Therefore, the action $H\curvearrowright \partial_FG$ is trivial. Since any $G$-boundary $\partial G$ is a $G$-factor of $\partial_FG$, it follows that $H$ acts trivially on $\partial G$. The claim follows by appealing to Proposition~\ref{prop:intromainsplittingC*}.
\end{proof}
\section{Intermediate Subalgebras of the Tensor Products}
\label{sec:intermediate}
This section is dedicated to proving Theorem~\ref{thm:introcontinuousfactortheorem}. Our strategy revolves around the interplay between the operator algebraic tools pitched together with the boundary actions. We have already established the operator algebraic splitting theorem in Section~\ref{sec:simpleinvariantC*-algebras}. We first show that the intermediate algebra must contain a piece of the boundary under our assumption.
\begin{proposition}
\label{prop:invariantinclusion}
Let $X$ be a $G$-space, and $\mathcal{B}$ be a unital $G$-$C^*$-algebra. Let $\mathcal{C}\subset C(X,\mathcal{B})$ be a $G$-invariant $C^*$-subalgebra such that  $$\Tilde{\mathcal{B}}\subset\langle \mathcal{C}, C(X)\rangle,$$
where $\tilde{\mathcal{B}}\subset \mathcal{B}$ is a $G$-invariant $C^*$-subalgebra and $\langle \mathcal{C}, C(X)\rangle$ denotes the $C^*$-subalgebra of $C(X,\mathcal{B})$ generated by $\mathcal{C}$ and $C(X)$. If there exists $\Gamma\le G$ such that \textcolor{red}{$\Gamma\curvearrowright X$} is strongly proximal and $\Gamma\curvearrowright\Tilde{\mathcal{B}}$ is trivial, then $\tilde{\mathcal{B}}\subset \mathcal{C}$.
\begin{proof}
Let $\varphi$ be a bounded linear functional on $C(X,\mathcal{B})$ such that $\varphi|_{\mathcal{C}}=0$. We show that $\varphi|_{\Tilde{\mathcal{B}}}=0$ from whence the claim will follow. Towards this end, let us write $\varphi=\sum_{i=1}^4c_i\varphi_i$, where $\varphi_i\in S(C(X,\mathcal{B}))$ and $c_i\in\mathbb{C}$, for each $i=1,2,3,4$. Now, since $\Gamma\curvearrowright X$ is strongly proximal, we can find a net $s_j\in \Gamma$ and some $x\in X$ such that
\[s_j.\left(\frac{1}{4}\sum_{i=1}^4\varphi_i|_{C(X)}\right)\xrightarrow[]{weak^*}\delta_x,~i=1,2,3,4.\]
Upon passing to a subnet four times if required, we can assume that $s_j.\varphi_i\xrightarrow[]{weak^*}\psi_i$ for each $i=1,2,3$ and $4$. It now follows that for $\psi:=\sum_{i=1}^4c_i\psi_i$, then $\psi|_{C(X)}=\left(\sum_{i=1}^4c_i\right)\delta_x$, where $\psi_i|_{C(X)}=\delta_x$. It is, of course, clear that $s_j.\varphi\xrightarrow[]{weak^*}\psi$. Moreover, since $\varphi(a)=0$ for all $a\in\mathcal{C}$ and $\mathcal{C}$ is $G$-invariant, we see that
\begin{equation}
 \label{eq:zero}
\psi(a)=\lim_j\varphi(s_j^{-1}a)=0,~\forall a\in\mathcal{C}.
\end{equation}
Now, let $b\in \Tilde{\mathcal{B}}$. Since $\Tilde{\mathcal{B}}\subset \langle \mathcal{C}, C(X)\rangle$, we can write
\begin{align*}
    b\approx \sum_{k=1}^ma_kh_k,~a_k\in\mathcal{C}, h_k\in C(X).
\end{align*}
Since $\Gamma\curvearrowright \tilde{\mathcal{B}}$ is trivial, we see that
\begin{align*}
    b=s_j^{-1}.b\approx \sum_{k=1}^m\left(s_j^{-1}.a_k\right)\left(s_j^{-1}.h_k\right).
\end{align*}
Applying $\varphi$ on both sides, we see that
\begin{align*}
    \varphi(b)&=\lim_j\varphi(s_j^{-1}.b)\\&\approx \lim_j\sum_{k=1}^m\varphi(\left(s_j^{-1}.a_k\right)\left(s_j^{-1}.h_k\right))\\&=\lim_j\sum_{k=1}^ms_j.\varphi(a_kh_k)\\&=\sum_{i=1}^4\sum_{k=1}^mc_i\psi_i(a_kh_k)\\&=\sum_{i=1}^4\sum_{k=1}^mc_i\psi_i(a_k)h_k(x)&\text{($\psi_i|_{C(X)}=\delta_x$)}\\&=\sum_{k=1}^m\left(\sum_{i=1}^4c_i\psi_i(a_k)\right)h_k(x)\\&=\sum_{k=1}^m\psi(a_k)h_k(x)\\&=0&\text{(equation~\eqref{eq:zero})}
\end{align*}
The claim follows.
\end{proof}
\end{proposition}
\begin{cor}
\label{cor:invariantsplitting}
Let $X$ be a $G$-boundary. Let $Y$ be a minimal $H$-space. Consider the product action $G\times H\curvearrowright X\times Y$. Let $\mathcal{C}$ be a $G\times H$-invariant $C^*$-subalgebra of $C(X,C(Y))=C(X)\otimes_{\text{min}} C(Y)$. Then, $\mathcal{C}=C(X', C(Y'))$, where $X'$ and $Y'$ are $G$ and $H$-factors of $X$ and $Y$ respectively. In other words, every $G\times H$-factor of $X\times Y$ splits as $X'\times Y'$ for some $G$-factor $X'$ of $X$ and some $H$-factor $Y'$ of $Y$.
\end{cor}
\begin{proof}
For the sake of simplicity, let us write $\mathcal{B}=C(Y)$.
Note that $$C(X)\subset \langle \mathcal{C}, C(X)\rangle \subset C(X,\mathcal{B}).$$
Applying Proposition~\ref{prop:intromainsplittingC*} for $\Gamma=G$, we see that $$\langle \mathcal{C}, C(X)\rangle=C(X,\Tilde{\mathcal{B}}),$$
where $\Tilde{\mathcal{B}}\subset\mathcal{B}$ is an $H$-invariant subalgebra and hence can be denoted by $\Tilde{\mathcal{B}}=C(Y')$ for some $H$-factor $Y'$ of $Y$. Note that $H\curvearrowright Y'$ is still minimal.
Using Proposition~\ref{prop:invariantinclusion}, we see that $\Tilde{\mathcal{B}}\subset\mathcal{C}$.  Thus, $$\Tilde{\mathcal{B}}\subset \mathcal{C} \subset C(X,\Tilde{\mathcal{B}}).$$ Again, applying Proposition~\ref{prop:intromainsplittingC*} for $\Gamma=H$, we see that $$\mathcal{C} =C(X',\Tilde{\mathcal{B}})=C(X', C(Y')).$$
\end{proof}
The strong proximality assumption in Corollary~\ref{cor:invariantsplitting} could not be dropped in general, as shown by the following example due to Zhengxing Lian. We thank him for allowing us to include it here.
\begin{example}
\label{ex:notproximal}
Let $G\curvearrowright X$ and $H\curvearrowright Y$ be irrational rotations defined by the same angle. Let $Z$ be defined as the $G\times H$-factor of $X\times Y$ by modding out the relation $(x,y)\sim (x',y')$ iff $d(x, x')=\frac{1}{2}$ and $d(y, y')=\frac{1}{2}$, unless $(x,y)=(x', y')$. Here, we write $X=[0,1]~mod~1$ and define $d(x, y)=\inf\{|x-y|, 1-|x-y|\}$ for any $x,y\in X$. To see $Z\vee X=X\times Y$, we check that the algebra generated by $C(Z)$ and $C(X)$ separate points in $X\times Y$. Indeed, for any $(x, y)\neq (x', y')\in X\times Y$, if $x\neq x'$, then we can use elements in $C(X)$ to separate $(x, y)$ from $(x', y')$; we may assume that $x=x'$ without loss of generality, then note that $(x, y)\not\sim(x, y')$ for any $y\neq y'$, hence elements in $C(Z)$ can be used to separate them. Similarly, we can show $Z\vee Y=X\times Y$. This means if $Z$ splits as a product, then $Z=X\times Y$. However, it is clear that $Z\neq X\times Y$.
\end{example}
\subsection{Intermediate Factor theorem}
\label{subsec:IFT}
We prove a continuous version of the intermediate factor theorem for a product of discrete groups acting on a product space, i.e., Theorem \ref{thm:introcontinuousfactortheorem}. This is analogous to the measurable version proved in \cite{bader2006factor}. Note that \cite[Theorem~1.9]{bader2006factor} is true for locally compact groups in general. We prove the following more general version in which only one boundary action for the product factor is assumed.
\begin{thm}
\label{thm:continuousfactortheorem}
Let $G=\Gamma_1\times \Gamma_2$ be a product of two discrete groups. Let $\Gamma_1\curvearrowright \partial\Gamma_1$ be a $\Gamma_1$-boundary and $\Gamma_2\curvearrowright Z$ be a continuous action on a compact Hausdorff space $Z$. Consider the product action $G\curvearrowright \partial\Gamma_1\times Z$. Let $G\curvearrowright X$ be a continuous action on a compact Hausdorff space. Moreover, assume that the diagonal actions $\Gamma_1\curvearrowright X\times \partial\Gamma_1$ and $\Gamma_2\curvearrowright X\times Z$ are both minimal. Then, every intermediate $G$-$C^*$-subalgebra $\mathcal{C}$ with
\[C(X)\subset\mathcal{C}\subset C(X)\otimes_{\text{min}} C(\partial\Gamma_1\times Z)\]
is a tensor product of the form $ C(X)\otimes_{\text{min}} C(B)$, where $B$ is a $G$-factor of $\partial\Gamma_1\times Z$ which splits as a product of a $\Gamma_1$-factor of $\partial\Gamma_1$ and a $\Gamma_2$-factor of $Z$.
\end{thm}
Clearly, Theorem \ref{thm:introcontinuousfactortheorem} is a particular case of the above theorem by taking $\Gamma_2\curvearrowright Z$ to be also a boundary action $\Gamma_2\curvearrowright\partial \Gamma_2$.

Let us briefly compare the assumptions above to their measurable counterparts in \cite[Theorem~1.9]{bader2006factor}. In there, $(X,\xi)$ is a measure-preserving $G=G_1\times G_2$-space such that the action $G_i\curvearrowright (X,\xi)$ is ergodic for each $i=1,2$. For $\mu=\mu_1\times\mu_2\in \text{Prob}(G_1)\times\text{Prob}(G_2)$, it follows from \cite[Corollary~3.2]{bader2006factor} that the corresponding Poisson boundary $(B,\nu_B)$ is of the form $(B_1,\nu_{B_1})\times (B_2,\nu_{B_2})$ such that $G_i\curvearrowright (B_{3-i},\nu_{B_{3-i}})$ is trivial. Moreover, since $G_i\curvearrowright (B_i,\nu_{B_i})$ is ergodic, it follows from \cite[Corollary~2.18]{bader2006factor} that the action $G_i\curvearrowright (X\times B_i,\xi\times\nu_{B_i})$ is ergodic.

Being ergodic in the measurable setup is equivalent to saying that the corresponding $L^{\infty}$-space has no non-trivial invariant WOT-closed ideals. Correspondingly, the topological analog is that the action is minimal, which means that the space of continuous functions does not have any non-trivial invariant closed ideals. This naturally leads us to assume that the action $\Gamma_i\curvearrowright X\times \partial\Gamma_i$ is minimal. Of course, the topological boundary is the natural counterpart in the continuous setup. Consequently, we can consider the above theorem as a topological variant of the IFT (\cite[Theorem~1.9]{bader2006factor}).
\begin{proof}[Proof of Theorem~\ref{thm:continuousfactortheorem}]
Let $\mathcal{C}_1=\left\langle\mathcal{C}, C(\partial\Gamma_1)\right\rangle$. Then, we see that
\[C(X)\otimes_{\text{min}} C(\partial\Gamma_1)\subset\mathcal{C}_1\subset\left(C(X)\otimes_{\text{min}} C(\partial\Gamma_1)\right)\otimes_{\text{min}} C(Z) \]
Note that $\Gamma_1\curvearrowright Z$ is trivial. Since $\Gamma_1\curvearrowright X\times \partial \Gamma_1$ is minimal and $\Gamma_1\curvearrowright Z$ is trivial, we deduce from Proposition~\ref{prop:intromainsplittingC*} that
\[\mathcal{C}_1=\left(C(X)\otimes_{\text{min}} C(\partial\Gamma_1)\right)\otimes_{\text{min}} C(B_2),\]
where $B_2$ is a $\Gamma_2$-factor of $Z$. In particular, this implies that
\[C(B_2)\subset \langle\mathcal{C}, C(\partial\Gamma_1)\rangle.\]
Since $\Gamma_1\curvearrowright B_2$ is trivial and $\Gamma_1\curvearrowright\partial \Gamma_1$ is a boundary action, using Proposition~\ref{prop:invariantinclusion}, we get that $C(B_2)\subset\mathcal{C}$. Therefore, it follows that
\[C(X)\otimes_{\text{min}} C(B_2)\subset\mathcal{C}\subset \left(C(X)\otimes_{\text{min}} C(B_2)\right)\otimes_{\text{min}} C(\partial\Gamma_1)\]
Since $\Gamma_2\curvearrowright X\times Z$ is minimal, using Proposition~\ref{prop:intromainsplittingC*} again, we see that
\[\mathcal{C}=\left(C(X)\otimes_{\text{min}} C(B_1)\right)\otimes_{\text{min}} C(B_2),\]
where $B_1$ is a $\Gamma_1$-factor of $\partial\Gamma_1$. Letting $B=B_1\times B_2$, the claim follows.
\end{proof}

Note that the assumptions of minimality of $\Gamma_1\curvearrowright X\times \partial\Gamma_1$ and $\Gamma_2\curvearrowright X\times Z$ force the diagonal action $G\curvearrowright X\times (\partial\Gamma_1\times Z)$ to be minimal. In general, we need the ambient space to be minimal (see Proposition~\ref{prop:condforsplitting} and the example following it).
\begin{remark}
The above proof as it is goes through for $G=\Gamma_1*\Gamma_2$ with exactly the same assumptions. We leave the details to the reader.
\end{remark}
Here, we present some examples to which Theorem \ref{thm:continuousfactortheorem} could be applied.
\begin{example}
Let $G=\Gamma_1\times\Gamma_2$ be a $C^*$-simple group. It follows from \cite[Theorem~1.4]{BKKO} that $\Gamma_1$ and $\Gamma_2$ are both $C^*$-simple. Let $\partial_F\Gamma_i$ be the corresponding Furstenberg boundaries for $i=1,2$. Using \cite[Lemma~5.2]{BKKO}, we see that the action $\Gamma_i\curvearrowright\partial_F\Gamma_i$ extends to an action $\Gamma\curvearrowright\partial_F\Gamma_i$ such that $\Gamma_{3-i}\curvearrowright\partial_F\Gamma_i$ is trivial for $i=1,2$. Now, for the diagonal action $G\curvearrowright\partial_F\Gamma_1\times \partial_F\Gamma_2$, any $G$-space $X$ under the additional assumption that $\Gamma_i\curvearrowright X\times\partial_F\Gamma_i$ is minimal, every intermediate $G$-$C^*$-subalgebra $\mathcal{C}$ with
\[C(X)\subset\mathcal{C}\subset C(X)\otimes_{\text{min}} C(\partial_F\Gamma_1\times \partial_F\Gamma_2)\]
is a tensor product.
\end{example}
Here is a more concrete example.
\begin{example}
    Let $\alpha: \mathbb{Z}\curvearrowright X$ be any minimal distal action, e.g., an irrational rotation on the unit circle. Assume that there is a surjective group homomorphism $\phi: \Gamma_1\times \Gamma_2\twoheadrightarrow \mathbb{Z}$ such that $\phi|_{\Gamma_i}$ are still surjective for $i=1,2$. Then we consider $\Gamma\curvearrowright X$ by $\alpha\circ \phi$.
    For example, take $\Gamma_1=F_2=\langle a,b\rangle$ and $\Gamma_2=F_2=\langle c,d\rangle$ and define the homomorphism  $\phi: \Gamma_1\times\Gamma_2\twoheadrightarrow \Gamma_1^{ab}\times \Gamma_2^{ab}\cong \mathbb{Z}^4=\langle \bar{a},\bar{b},\bar{c},\bar{d}\rangle\overset{\pi}{\twoheadrightarrow}\mathbb{Z}$ as the composition of taking abelization and $\pi$, where $\pi$ is defined by $\pi(n_1\bar{a}+n_2\bar{b}+n_3\bar{c}+n_4\bar{d})=n_1+n_3$ for all $n_1,\ldots, n_4\in\mathbb{Z}$.
    Note that as proximal actions, $\Gamma_i\curvearrowright \partial_F\Gamma_i$ are minimal and weakly mixing, see e.g., \cite[2nd paragraph on P. 176]{KLbook}, hence they are disjoint from the minimal distal actions $\Gamma_i\curvearrowright X$. Therefore, $\Gamma_i\curvearrowright X\times \partial_F\Gamma_i$ are minimal for $i=1,2$.
\end{example}
For a product group, the corresponding $\mu$-boundary (for a product measure $\mu$) is the product of $\mu$-boundaries as mentioned above (and shown in \cite[Corollary~3.2]{bader2006factor}). The topological analog of this result does not hold. 
\begin{proposition}
\label{prop:notproductboundary}
Let $G$ and $H$ be two discrete groups. Let $\partial_F(G\times H)$ denote the corresponding Furstenberg boundary. If $\partial_F(G\times H)$ is a product space equipped with the product action, then either $G$ or $H$ is amenable. 
\begin{proof}
Assume that $\partial_F(G\times H)$ is a product space equipped with the product action. Given any $G$-boundary $X$, the action $G\curvearrowright X$ can be extended to an action $G\times H \curvearrowright X$ such that $H\le \text{Ker}(G\times H\curvearrowright X)$ (see \cite[Lemma~5.2]{BKKO}). Therefore, it follows that 
\[C\left(\partial_F(G\times H)\right)^G=C(\partial_FH)\text{ and } C\left(\partial_F(G\times H)\right)^H=C(\partial_FG).\]
Consequently, we see that $\partial_F(G\times H)=\partial_FG\times \partial_FH$. Since $\partial_F(G\times H)$ is an extremally disconnected space (see \cite[Proposition~2.4]{BKKO}), we see that $\partial_FG\times \partial_FH$ is an extremally disconnected space. Using \cite[Corollary~3.15]{arhangel2011study}, we obtain that either one of $\partial_FG$ or $\partial_FH$ is finite. Without any loss of generality, let us assume that $\partial_FG$ is finite. Since $G\curvearrowright \partial_FG$ is strongly proximal, it follows that $\partial_FG$ is trivial. It is now a consequence of \cite[Theorem~3.1, Chapter~3]{Prox} that $G$ is amenable. 
\end{proof}
\end{proposition}

We end this section by presenting a different proof of Theorem~\ref{thm:continuousfactortheorem} under the weaker assumption of proximality instead of strong proximality. This was kindly communicated to us by Hanfeng Li. An action $G\curvearrowright X$ is called proximal if for every pair of points $(x, y)\in X\times X$, there exists a net $s_i\in G$ such that 
\[\lim_is_ix=\lim_i s_iy.\]
\begin{thm}
\label{thm:improvcontinuousfactortheorem}
Let $G=\Gamma_1\times \Gamma_2$ be a product of two discrete groups. Let $\Gamma_1\curvearrowright \partial_P\Gamma_1$ be a proximal action and $\Gamma_2\curvearrowright Z$ be a continuous action on a compact Hausdorff space $Z$. Consider the product action $G\curvearrowright \partial_P\Gamma_1\times Z$. Let $G\curvearrowright X$ be a continuous action on a compact Hausdorff space. Moreover, assume that the diagonal actions $\Gamma_1\curvearrowright X\times \partial_P\Gamma_1$ and $\Gamma_2\curvearrowright X\times Z$ are both minimal. Then, every intermediate $G$-$C^*$-subalgebra $\mathcal{C}$ with
\[C(X)\subset\mathcal{C}\subset C(X)\otimes_{\text{min}} C(\partial_P\Gamma_1\times Z)\]
is a tensor product of the form $ C(X)\otimes_{\text{min}} C(B)$, where $B$ is a $G$-factor of $\partial_P\Gamma_1\times Z$ which splits as a product of a $\Gamma_1$-factor of $\partial_P\Gamma_1$ and a $\Gamma_2$-factor of $Z$.
\begin{proof}
Write $W=Spec(\mathcal{C})$, the Gelfand spectrum of $\mathcal{C}$. Then we have successive $G$-factor maps $X\times \partial_P\Gamma_1\times Z\rightarrow W\rightarrow X$. Write $W=\frac{X\times \partial_P\Gamma_1\times Z}{\sim}$ for a closed $G$-invariant equivalence relation $\sim$ on $X\times \partial_P\Gamma_1\times Z$. 
We aim to show that $W=X\times B_1\times B_2$ for some $\Gamma_1$-factor $\partial_P\Gamma_1\rightarrow B_1$ and some $\Gamma_2$-factor $Z\rightarrow B_2$.

Take any $(x_1, y_1, z_1)$ and $(x_2, y_2, z_2)$ in $X\times\partial_P\Gamma_1\times Z$ such that $(x_1,y_1,z_1)\sim (x_2,y_2,z_2)$. Since $W$ factors onto $X$, we deduce that $x_1=x_2$.

 Since $\Gamma_1\curvearrowright\partial_P\Gamma_1$ is proximal, there exists some net $s_i\in \Gamma_1$ such 
that $\lim_is_iy_1=\lim_is_iy_2=:y'\in \partial_P\Gamma_1$. Then, after passing to some subnet if necessary, we may assume $\lim_is_ix_1=:x'\in X$. Then since $\sim$ is closed and $\Gamma_1$-invariant, and $\Gamma_1$ acts trivially on $Z$ in the definition of product action $G\curvearrowright \partial_P\Gamma_1\times Z$, we deduce that $(x', y', z_1)\sim (x',y', z_2)$. Since $\Gamma_1\curvearrowright X\times \partial_P\Gamma_1$ is minimal and $\Gamma_1$ acts trivially on $Z$, it is clear that the above implies that 
\begin{align}\label{eq: 2nd coordinate can change}
(x, y, z_1)\sim (x,y, z_2)~ \text{for all}~ (x, y)\in X\times\partial_P\Gamma_1.
\end{align}
In particular, we get that $(x_1,y_1,z_1)\sim (x_1, y_1,z_2)$. Hence, $(x_1, y_2, z_2)\sim (x_1, y_1,z_1)\sim(x_1, y_1, z_2)$. Since $\Gamma_2\curvearrowright X\times Z$ is minimal and $\Gamma_2$ acts trivially on $\partial_P\Gamma_1$ from definition of the product action $G\curvearrowright\partial_P\Gamma_1\times Z$, this implies that actually 
\begin{align}\label{eq: 3rd coordinate can change}
(x, y_2, z)\sim (x, y_1, z)~\text{for all}~ (x, z)\in X\times Z.
\end{align}
Finally, we may define two closed equivalence relations, $\sim_1$ on $\partial_P\Gamma_1$ and $\sim_2$ on $Z$ as follows:

Define $y_1\sim_1y_2$ if $(x, y_1, z)\sim (x, y_2, z)$ for all $(x, z)\in X\times Z$;

Define $z_1\sim_2 z_2$ if $(x, y, z_1)\sim (x, y, z_2)$ for all $(x, y)\in X\times \partial_P\Gamma_1$.

Then one can check that $\sim_1$ is $\Gamma_1$-invariant and $\sim_2$ is $\Gamma_2$-invariant. Moreover, using \eqref{eq: 2nd coordinate can change} and \eqref{eq: 3rd coordinate can change}, it is clear that $(x_1,y_1,z_1)\sim (x_2, y_2, z_2)$ iff $x_1=x_2$, $y_1\sim_1 y_2$ and $z_1\sim_2z_2$. Hence we may just set $B_1=\frac{\partial_P\Gamma_1}{\sim_1}$ and $B_2=\frac{Z}{\sim_2}$ to finish the proof.
\end{proof}
\end{thm}

\section{On the Assumption of \texorpdfstring{$G$}{G}-Simplicity of \texorpdfstring{$C(X,\mathcal{A})$}{}}
\label{sec:simplicityassumption}
It is known that the ideals in the underlying $C^*$-algebra obstruct the splitting theorem. This has been observed in \cite{zacharias2001splitting}. Also, as mentioned in Remark~\ref{rem:ideal}, with the additional $G$-action, the underlying algebra must be $G$-simple. Moreover, the assumptions of Theorem~\ref{thm:continuousfactortheorem} implicitly force the ambient tensor product algebra to be $G$-simple. In this section, we show that such an assumption is necessary.

We begin by giving a condition that forces the ambient tensor product to be $G$-simple for the cases when the underlying algebra is abelian.
\begin{proposition}
Let $G\curvearrowright X$ be a minimal action on a compact Hausdorff space and $\mathcal{A}$ be a unital $G$-$C^*$-algebra such that $\mathcal{A}$ itself is simple. Then $C(X)\otimes_{\text{min}}\mathcal{A}$ is $G$-simple.
\begin{proof}
We begin by identifying $C(X)\otimes_{\text{min}} \mathcal{A}$ with $C(X, \mathcal{A})$. Under this identification, for $f\in C(X,\mathcal{A})$, $$(s.f)(x)=\sigma_s(f(s^{-1}x)),~s\in G, x\in X.$$
Let $I$ be a nonzero closed $G$-invariant two-sided ideal in $C(X)\otimes_{\text{min}} \mathcal{A}$. We aim to show $I=C(X)\otimes_{\text{min}} \mathcal{A}$.

For any $x\in X$, write $I_x=\{f(x)\mid f\in I\}\subset \mathcal{A}$.
Note that $I_{sx}=\sigma_s(I_x)$. Indeed, we observe that \begin{align*}
a \in I_x&\Leftrightarrow \exists~f\in I,~\text{s.t.}~ f(x)=a\\
&\Leftrightarrow (sf)(sx)=\sigma_s(f(x))=\sigma_s(a)~\\
&\Leftrightarrow \sigma_s(a)\in I_{sx}~\text{since $sf\in I$}.
\end{align*}
Observe that $I_x\lhd\mathcal{A}$ is an ideal. Since $\mathcal{A}$ is simple, we have either $\overline{I_x}=\{0\}$ or $\overline{I_x}=\mathcal{A}$. Define $X_0=\{x\in X\mid \overline{I_x}=\{0\}\}$ and $X_1=\{x\in X\mid \overline{I_x}=\mathcal{A}\}$. Then $X=X_0\sqcup X_1$. Observe that $X_0$ is a closed $G$-invariant subset in $X$. Indeed, $G$-invariance follows from $I_{sx}=\sigma_s(I_x)$. For the closeness, take any $x_n\in X_0$ and assume that $x_n\rightarrow x$, then $f(x_n)=0_{\mathcal{A}}$ and by taking the limit, we get that $f(x)=0_{\mathcal{A}}$ for all $f\in I$. Hence $x\in X_0$.

Since $G\curvearrowright X$ is minimal, we get two possibilities. Either, $X_0=\emptyset$ and hence we get $X=X_1$. Since $I\neq (0)$, we get $I_x\neq 0$ for some $x\in X$ and hence $X=X_1$.
It follows directly from Lemma \ref{lem: Naimark} that $I=C(X, \mathcal{A})$ in this case. (The proof uses the same ingredients as the proof of Lemma \ref{lem:smallparts}).
\end{proof}
\end{proposition}
If $\mathcal{A}$ is not simple, then it is possible that $C(X,\mathcal{A})$ fails to be $G$-simple even if $\mathcal{A}$ is $G$-simple. Indeed, just take a minimal action $\beta: G\curvearrowright Y$ such that it is not disjoint from $\alpha: G\curvearrowright X$, e.g., $\beta=\alpha$ and set $\mathcal{A}=C(Y)$. 

We now show that the existence of $G$-invariant ideals $I\triangleleft C(X,\mathcal{A})$ obstructs the splitting result. We begin with the special case when $\mathcal{A}$ is a unital commutative $G$-$C^*$-algebra.
\begin{proposition}
\label{prop:condforsplitting}
Let $G\curvearrowright X$ and $G\curvearrowright Y$ be two minimal actions. Assume that the diagonal action $G\curvearrowright X\times Y$ is not minimal and let $\emptyset\neq Z\subsetneq X\times Y$ be a $G$-invariant closed subset. Let $I\lhd C(X\times Y)$ be defined by $f\in I$ iff $f|_Z=0$. Set $W=C(Y)+I$. Then, the following hold.
\begin{enumerate}
    \item $W$ is a $G$-invariant $C^*$-algebra in $C(X\times Y)$ containing $C(Y)$.
    \item $W=C(X\times Y)\Leftrightarrow Z=\{(\phi(y),y): y\in Y \}$ for some $G$-equivariant surjection $\phi: Y\to X$.
    \item If $W= C(X'\times Y)$ for some $G$-factor $X'$ of $G\curvearrowright X$, then $X'=X$ and there is a continuous surjection $\phi: Y\to X$ such that $Z=\{(\phi(y), y):~y\in Y\}$.
    \item If $Z=\{(\phi(y), y), (\psi(y), y): ~y\in Y\}$ for two non-equal conjugacies $\phi,\psi: Y\approx X$, then $W$ is an intermediate $G$-$C^*$-subalgebra between $C(Y)$ and $C(X\times Y)$ which does not split.
    \end{enumerate}
\end{proposition}
\begin{proof}
(1) is clear by \cite[Theorem 3.1.7]{murphy_book}.

(2) $\Leftarrow$: By Stone-Weierstrass theorem, it suffices to check $W$ separates points in $X\times Y$. Take any $(x, y)\neq (x',y')$. If $y\neq y'$, we may use $C(Y)\subset W$ to separate them. So we may assume $y=y'$ and $x\neq x'$. Then from the definition of $Z$, either $(x, y)\in Z\not\ni (x', y)$, or $(x', y)\in Z\not\ni (x, y)$ or $(x, y)\not\in Z\not\ni (x', y)$. In the first two cases, we can clearly use $I$ to separate them. For the last case, observe that we may start with the continuous function $a: Z\cup \{(x_1, y), (x_2, y)\}\rightarrow \mathbb{R}$ defined by $a|_Z=0, a(x_i, y)=i,~i=1, 2$ and apply the Tietze extension theorem (by noting that $Z\cup \{(x_1, y), (x_2, y)\}\subset X\times Y$ is closed.) to get an $f\in I$ such that $f$ separates the two points $(x, y), (x', y)$.

(3) is, a priori, stronger than the forward implication in (2), so we just need to prove (3).

Assume that $W=C(X'\times Y)$ for some factor $G\curvearrowright X'$ of $G\curvearrowright X$ with $\pi: X\rightarrow X'$ the factor map.

$C(X'\times Y)$ has Gelfand spectrum $X'\times Y$, in other words, the evaluation map on $X'\times Y$ gives us a well-defined algebraic homomorphism on $C(X'\times Y)=W\supseteq I$. In particular, for any $y\in Y$, and any $x_1\neq x_2\in X$ with $\pi(x_1)=\pi(x_2)$, then we have $f(x_1, y)=f(x_2, y)$ for all $f\in I$. This implies that $(x_1, y)\in Z\ni (x_2, y)$.

To see this, assume that $(x_1, y)\not\in Z$, then $\exists f\in I$ with $f(x_1, y)\neq 0$. Hence $f(x_2, y)=f(x_1, y)\neq 0$, thus $(x_2, y)\not\in Z$. Since $(x_1, y)\neq (x_2, y)$, we may find some $f\in I$ with $f(x_1, y)\neq f(x_2, y)$. Indeed, just start with the continuous function $a: Z\cup \{(x_1, y), (x_2, y)\}\rightarrow \mathbb{R}$ defined by $a|_Z=0, a(x_i, y)=i,~i=1, 2$ and apply the Tietze extension theorem (by noting that $Z\cup \{(x_1, y), (x_2, y)\}\subset X\times Y$ is closed.) to get such an $f\in I$. This gives us a contradiction.

Therefore, we have proved that $\pi^{-1}(x')\times Y \subset Z$ for all $x'\in X'$ with $|\pi^{-1}(x')|>1$. But
$G\curvearrowright X$ is minimal implies that $X\times Y\subset Z$ once we have some $x'\in X'$ with $|\pi^{-1}(x')|>1$, which gives us a contradiction since $Z$ is a proper subset in $X\times Y$. This shows that in fact  for all $x'\in X'$, $|\pi^{-1}(x')|=1$; equivalently, $\pi$ is 1-1. Hence $X=X'$ and $W=C(X\times Y)$. We are left to show that $Z$ is of the desired form.

For this, just observe that since $W$ separates points in $X\times Y$, we know that for each $y\in Y$, $|\{x\in X: ~(x, y)\in Z\}|\leq 1$. But $G\curvearrowright Y$ is minimal and $Z\neq\emptyset$, hence in fact we have $|\{x\in X: (x, y)\in Z\}|=1$ for all $y\in Y$. Then it is not hard to see from this fact, and $Z$ is $G$-invariant that $Z$ is of the desired form.

(4) This is clear by (3). 
\end{proof}
Of course, more examples of $Z$ in the above proposition can be constructed using item (3) there to give non-split $G$-invariant intermediate subalgebras.
\begin{example}
Let $\Gamma$ be a non-amenable group, and $\partial_F\Gamma$, the associated Furstenberg boundary. Let $X$ be a minimal metrizable $\Gamma$-space with the property that $X\times\partial_F\Gamma$ is not minimal. Let $Z$ be a $\Gamma$-invariant closed subset of $X\times\partial_F\Gamma$. Denote by $I$ the $\Gamma$-invariant closed ideal defined by
\[I=\{f\in C(X\times\partial_F\Gamma): f|_{Z}=0\}.\]
It follows from Proposition~\ref{prop:condforsplitting} that $W=C(X)+I$ is an intermediate $\Gamma$-$C^*$-subalgebra. We claim that $W$ does not split. Indeed, if $W$ were to split, by Proposition~\ref{prop:condforsplitting}-(3), it must be the case that there exists $\Gamma$-equivariant continuous surjection $\phi: X\to\partial_F\Gamma$ such that $Z=\{(\phi(x),x):x\in X\}$. Since $X$ is metrizable, we see that $\partial_F\Gamma$ is metrizable. However,  $\partial_F\Gamma$ is never metrizable unless it is a one-point space (see \cite[Corollary~3.17]{kalantar_kennedy_boundaries}), in which case $\Gamma$ has to be amenable. The claim follows.
\end{example}
\begin{proposition}
\label{prop:hom}
Let $\mathcal{A}$ and $\mathcal{B}$ be $G$-simple unital $C^*$-algebras. Assume that $\mathcal{A}\otimes_{\text{min}}\mathcal{B}$ is not $G$-simple. Let $I\triangleleft \mathcal{A}\otimes_{\text{min}}\mathcal{B}$ be a nontrivial $G$-invariant closed two-sided ideal. Let $$\mathcal{C_{\mathcal{A}}}=\mathcal{A}+I\text{ and } \mathcal{C_{\mathcal{B}}}=\mathcal{B}+I.$$
Then, the following holds true:
\begin{enumerate}
\item $\mathcal{A}\subset\mathcal{C_{\mathcal{A}}}\subset\mathcal{A}\otimes_{\text{min}}\mathcal{B}$ is an intermediate $C^*$-algebra. Similarly, $\mathcal{B}\subset\mathcal{C_{\mathcal{B}}}\subset\mathcal{A}\otimes_{\text{min}}\mathcal{B}$ is an intermediate $C^*$-algebra.
\item If $\mathcal{C}_{\mathcal{A}}$ splits, then there exists an injective $*$-homomorphism $\Phi:\mathcal{B}\to\mathcal{A}$. Similarly, if $\mathcal{C}_{\mathcal{B}}$ splits, then there exists an injective $*$-homomorphism $\Phi:\mathcal{A}\to\mathcal{B}$.
\end{enumerate}
\begin{proof}
Since $I\cap\mathcal{A}$ is a $G$-invariant ideal of $\mathcal{A}$, it follows from the $G$-simplicity of $\mathcal{A}$ that $I\cap\mathcal{A}=\{0\}$. A similar argument shows that $I\cap\mathcal{B}=\{0\}$.

Assume that $\mathcal{C}_{\mathcal{A}}$ splits, i.e., $\mathcal{C}_{\mathcal{A}}=\mathcal{A}\otimes_{\text{min}}\tilde{\mathcal{B}}$ for some $G$-invariant subalgebra $\tilde{\mathcal{B}}\subset\mathcal{B}$. We first claim that $\mathcal{C}_{\mathcal{A}}=\mathcal{A}\otimes_{\text{min}}\mathcal{B}$. Write $\mathcal{A} + I = \mathcal{C}_{\mathcal{A}} = \mathcal{A}\otimes_{\text{min}}\tilde{\mathcal{B}}$. Given that $I \cap\mathcal{A} = 0$, it must be the case that $\tilde{\mathcal{B}}\ne\mathbb{C}$. Also, we know that $\mathbb{E}_{\phi}(I)\subseteq \tilde{\mathcal{B}}$ for all $\phi\in S(\mathcal{A})$. Now choose $b\in \tilde{\mathcal{B}}\setminus\mathbb{C}$. We know that there
exists some $a \in \mathcal{A} $ and $i \in I$ so that $a \otimes 1+i = 1 \otimes b$. Fixing some $\phi \in S(\mathcal{A})$, we may apply
$\mathbb{E}_{\phi} : \mathcal{A} \otimes_{\text{min}} \mathcal{B} \to \mathcal{B}$ to get $\mathbb{E}_{\phi}(i) = b-\phi(a)$. Observe that this value is nonzero, as we were assuming that $b \not\in \mathbb{C}$. Also, we have
$$g\mathbb{E}_{\phi}(i) = gb- \phi(a) = \mathbb{E}_{g\phi}(gi).$$
Given that these values always lie in $\tilde{\mathcal{B}}$, and the fact that all $\mathbb{E}_{g\phi}$ are $\mathcal{B}$-bimodule maps, it
must be the case that the ideal generated by $\mathbb{E}_{\phi}(i)$ is still contained in $\tilde{\mathcal{B}}$. But this ideal is $\mathcal{B}$ itself. So, $\mathcal{B}=\tilde{\mathcal{B}}$ and the claim follows.

In particular, $\mathcal{B}\subset\mathcal{C}_{\mathcal{A}}$. Each $b\in\mathcal{B}$ can be written as $b=i_b+a_b$ for some $i_b\in I$ and $a_b\in\mathcal{A}$. Define $\Phi: \mathcal{B}\to\mathcal{A}$ by $b\mapsto a_b$. Let us first check the well-definedness. If $i_b+a_b=i_b'+a_b'$ for some $i_b,i_b'\in I$ and $a_b,a_b'\in\mathcal{A}$, then we see that $i_b-i_b'=a_b-a_b'\in I\cap\mathcal{A}$. Since $I\cap\mathcal{A}=\{0\}$, we obtain that $a_b=a_b'$ and $i_b=i_b'$. This shows that every element in $\mathcal{B}$ can be uniquely written as a sum of two elements from $I$ and $\mathcal{A}$. Also, for $b,\Tilde{b}\in\mathcal{B}$, we see that
\begin{align*}
\Phi(b\tilde{b})&=\Phi\left((i_b+a_b)(i_{\tilde{b}}+a_{\tilde{b}})\right)\\&=\Phi(i_{b}i_{\tilde{b}}+i_ba_{\tilde{b}}+a_bi_{\tilde{b}}+a_ba_{\tilde{b}})\\&=a_ba_{\tilde{b}}\\&=\Phi(b)\Phi(\Tilde{b})
\end{align*}
Similarly,
\[\Phi(b^*)=\Phi(i_b^*+a_b^*)=a_b^*=\Phi(b)^*,~b\in\mathcal{B}.\]
Moreover, since $I$ is $G$-invariant (we denote the $G$-action by $\sigma$), it also follows that
\[\Phi(\sigma_s(b))=\Phi(\sigma_s(i_b)+\sigma_s(a_b))=\sigma_s(a_b)=\sigma_s\left(\Phi(b)\right),~b\in\mathcal{B},~s\in G.\]
All that remains to be shown now is that $\Phi$ is injective. Towards that end assume that $\Phi(b)=\Phi(\tilde{b})$ for some $b,\Tilde{b}\in\mathcal{B}$. It then follows that $a_b=a_{\tilde{b}}$. Now,
\[b-\Tilde{b}=i_b+a_b-i_{\tilde{b}}-a_{\Tilde{b}}=i_b-i_{\tilde{b}}\in I\cap\mathcal{B}.\]
Since $I\cap\mathcal{B}=\{0\}$,  we see that $b=\Tilde{b}$. This shows that $\Phi$ is an injective $*$-homomorphism. A similar argument works for $\mathcal{C}_{\mathcal{B}}$.
\end{proof}
\end{proposition}
The following is now obvious.
\begin{thm}
\label{thm:hom}
Let $\mathcal{A}=C(X)$ and $\mathcal{B}$ be $G$-simple unital $C^*$-algebras admitting faithful states. Assume that $\mathcal{A}\otimes_{\text{min}}\mathcal{B}$ is not $G$-simple. Also, assume that $\mathcal{B}$ is non-commutative. Then, there exists an intermediate algebra $\mathcal{C}_{\mathcal{A}}$ with $C(X)\subsetneq \mathcal{C}_{\mathcal{A}}\subset C(X)\otimes_{\text{min}}\mathcal{B}$ such that $\mathcal{C}_{\mathcal{A}}$ does not split.
\begin{proof}
Let $I\triangleleft \mathcal{A}\otimes_{\text{min}}\mathcal{B}$ be a nontrivial $G$-invariant closed two-sided ideal. Let $$\mathcal{C_{\mathcal{A}}}=\mathcal{A}+I.$$
If $\mathcal{C_{\mathcal{A}}}$ were to split, there exists an injective $*$-homomorphism $\Phi:\mathcal{B}\to C(X)$. This is a contradiction since $\mathcal{B}$ is non-commutative.
\end{proof}
\end{thm}

\begin{example}
Let $G=\mathbb{Z}\rtimes \frac{\mathbb{Z}}{2\mathbb{Z}}=\langle s\rangle\rtimes \langle t\rangle=\langle s,t\mid tst^{-1}=s^{-1},t^2=e\rangle$ be the infinite Dihedral group. Then consider any minimal action $G\curvearrowright X$ such that the subaction $\mathbb{Z}=\langle s\rangle\curvearrowright X$ is not minimal. Such an action $G\curvearrowright X$ exists and is, in fact, conjugate to an induced action, see, e.g., \cite[proposition 5.1]{jiang2023continuous}. Let us briefly recall the construction.

Let $\mathbb{Z}=\langle s\rangle\curvearrowright X_0$ be any minimal action. Set $X=G/{\mathbb{Z}}\times X_0$ and the action $G\curvearrowright X=G/{\mathbb{Z}}\times X_0$ is defined by $g(g'\mathbb{Z}, x)=(gg'\mathbb{Z}, \delta(g, g'\mathbb{Z})x)$ for any $g, g'\in G$
and $x\in X_0$. Here, $\delta: G\times G/{\mathbb{Z}}\rightarrow \mathbb{Z}$ denotes the natural cocycle associated to the natural
lifting $L: G/{\mathbb{Z}}\rightarrow G$ defined by $L(\mathbb{Z})=e$ and $L(t\mathbb{Z})=t$. More precisely, we have
\begin{align*}
\delta(s^n, \mathbb{Z})=s^n,~~
\delta(s^n, t\mathbb{Z})=s^{-n},~~
\delta(s^nt, \mathbb{Z})=s^{-n},~~
\delta(s^nt,t\mathbb{Z})=s^n.
\end{align*}
Next, we set $\mathcal{A}=M_2(\mathbb{C})\oplus M_2(\mathbb{C})$ and consider the flip action $\alpha: \frac{\mathbb{Z}}{2\mathbb{Z}}\curvearrowright\mathcal{A}$ by swapping the coordinates. This gives us an action $\pi: G\twoheadrightarrow \frac{\mathbb{Z}}{2\mathbb{Z}}=\langle t\rangle\curvearrowright \mathcal{A}$ by composing $\alpha$ with the natural quotient homomorphism $G\twoheadrightarrow G/{\mathbb{Z}}\cong \frac{\mathbb{Z}}{2\mathbb{Z}}$. Note that $G\curvearrowright \mathcal{A}$ is $G$-simple. 

Indeed, let $I\lhd \mathcal{A}$ be a $G$-invariant ideal. First, note that every ideal in $\mathcal{A}=M_2(\mathbb{C})\oplus M_2(\mathbb{C})$ is of the form $0, M_2(\mathbb{C})\oplus 0_2, 0_2\oplus M_2(\mathbb{C})$ or $\mathcal{A}$ since $M_2(\mathbb{C})$ is simple. Then since $I$ is $\langle t\rangle$-invariant, we get $I=0$ or $\mathcal{A}$. 

We are left to show that the diagonal action $G\curvearrowright C(X)\otimes_{\text{min}} \mathcal{A}$ is not $G$-simple.  Recall that by identifying $C(X)\otimes_{\text{min}} \mathcal{A}$ with $C(X,\mathcal{A})$, the diagonal action is given by $(g\xi)(g')=\pi_g(\xi(g^{-1}g'))$ for all $g, g'\in G$. Note that $\pi_{s^n}=id$ and $\pi_{ts^n}=\pi_t$ is the flip automorphism. Define $$I=\{f\in C(X,\mathcal{A})\mid~ f|_{\mathbb{Z}\times X_0}=M_2(\mathbb{C})\oplus 0_2, ~f|_{t\mathbb{Z}\times X_0}=0_2\oplus M_2(\mathbb{C})\}.$$ It is routine to check that $I$ is a $G$-invariant non-zero proper ideal. To show the $G$-invariance, we do computations.
Take any $\xi\in I$ and $x_0\in X_0$. Let us check $g\xi\in I$ for any $g\in G$. \begin{align*}
(s^{-n}\xi)(\mathbb{Z}, x_0)&=\pi_{s^{-n}}(\xi(s^n(\mathbb{Z}, x_0))\\&=\xi(\mathbb{Z},\delta(s^n,\mathbb{Z})x_0)=\xi(\mathbb{Z}, s^nx_0)\in M_2(\mathbb{C})\oplus 0_2,\\
(s^{-n}\xi)(t\mathbb{Z}, x_0)&=\pi_{s^{-n}}(\xi(s^n(t\mathbb{Z}, x_0))\\&=\xi(t\mathbb{Z},\delta(s^n,t\mathbb{Z})x_0)\\&=\xi(t\mathbb{Z}, s^{-n}x_0)\in 0_2\oplus M_2(\mathbb{C}),\\
(ts^{-n}\xi)(\mathbb{Z}, x_0)&=\pi_{ts^{-n}}(\xi(s^nt(\mathbb{Z}, x_0))=\pi_t(\xi(t\mathbb{Z},\delta(s^nt,\mathbb{Z})x_0))\\
&=\pi_t(\xi(t\mathbb{Z}, s^{-n}x_0)\in \pi_t(0_2\oplus M_2(\mathbb{C}))=M_2(\mathbb{C})\oplus 0_2,\\
(ts^{-n}\xi)(t\mathbb{Z}, x_0)&=\pi_{ts^{-n}}(\xi(s^nt(t\mathbb{Z}, x_0))=\pi_t(\xi(\mathbb{Z},\delta(s^nt,t\mathbb{Z})x_0))\\
&=\pi_t(\xi(\mathbb{Z}, s^nx_0)\in \pi_t(M_2(\mathbb{C})\oplus 0_2)=0_2\oplus M_2(\mathbb{C}).
\end{align*}
\end{example}
Consequently, we obtain the following for the case of $C(X,\mathcal{A})$.
\begin{cor}
Let $Y$ be a $G$-minimal compact Hausdroff space  and $\mathcal{A}$ be a $G$-simple $C^*$-algebra. Let $I$ be a $G$-invariant closed ideal inside $C(Y,\mathcal{A})$. Let $\mathcal{C}_I=C(Y)+I$. Then, the following hold.
\begin{enumerate}
    \item $\mathcal{C}_I$ is a $G$-invariant $C^*$-algebra in $C(Y, \mathcal{A})$ containing $C(Y)$.
    \item $\mathcal{C}_I=C(Y,\mathcal{A}) \text{ iff } \mathcal{A}\text{ is of the form
 } C(X)\text{ for some $G$-space $X$}$ and there exists a non-empty closed $G$-invariant subset $Z\subset X\times Y$ such that $$I=\{f\in C(X\times Y): f|_{Z}=0\},$$ and $Z=\{(\phi(y),y): y\in Y \}$ for some $G$-equivariant continuous surjection $$\phi: Y\to X.$$
    \item If $\mathcal{C}_I= C(Y,\mathcal{B})$ for some $G$-invariant unital $G$-simple $C^*$-subalgebra $\mathcal{B}\subset\mathcal{A}$, then $\mathcal{B}$ is a unital commutative $C^*$-subalgebra of $C(Y)$.
    \end{enumerate}
\begin{proof}
$(1)$ is clear. Let us now prove $(2)$. Assume that $\mathcal{C}_I=C(Y,\mathcal{A})$. It follows from Proposition~\ref{prop:hom} that there exists a $G$-equivariant injective $*$-homomorphism $\Phi:\mathcal{A}\to C(Y)$. Therefore, $\mathcal{A}=C(X)$ for some $G$-space $X$. Therefore, $C(Y,\mathcal{A})=C(Y\times X)$. Consequently, there exists a non-empty closed $G$-invariant subset $Z\subset X\times Y$ such that $I=\{f\in C(X\times Y): f|_{Z}=0\}$. The existence of a $G$-equivariant surjection $\phi: Y\to X$ such that $Z=\{(\phi(y),y): y\in Y\}$ follows from Proposition~\ref{prop:condforsplitting}. Also, the reverse direction follows from the second claim of Proposition~\ref{prop:condforsplitting}. To prove $(3)$, we observe that if $\mathcal{C}_I= C(Y,\mathcal{B})$ for some $G$-invariant unital $C^*$-subalgebra $\mathcal{B}\subset\mathcal{A}$, then it follows from Proposition~\ref{prop:hom} that $\mathcal{B}$ is a commutative unital $C^*$-subalgebra of $C(Y)$ and hence, is of the form $C(X')$ for some $G$-space $X'$. The implication is now evident by the third claim of Proposition~\ref{prop:condforsplitting}.
\end{proof}
\end{cor}

\begin{example}[Finite dimensional spaces]
Let $X$ be a finite transitive $G$-space, let us write it as $\{x_1, x_2,\ldots,x_d\}$ for distinct points $x_i$, i.e., $x_i\ne x_j$ for $i\ne j$. Let
\[Z=\{(x_i,x_j): i\ne j\}\]
Observe that $Z$ is a closed $G$-invariant subset of $X\times X$.
Let $I=\{f\in C(X\times X): f|_{Z}=0\}$. Consider
\[\mathcal{A}_I= C(X)+I\]
Thus $\mathcal{A}_I$ is a $G$-invariant intermediate subalgebra by \cite[Theorem 3.1.7]{murphy_book}.\\
\textit{Claim:} $\mathcal{A}_I$ splits iff $|X|=2$.\\
Assume that $|X|=d>2$. For any $\sigma \in\text{Perm}(X)$, $\{(\sigma(x_i), (x_i)):i=1,2,\ldots d\}$ has exactly $d$ elements. On the other hand, $Z$ has exactly $d(d-1)$ elements. Since $d>2$, $$Z\ne \{(\sigma(x_i), (x_i)):i=1,2,\ldots d\}$$
Therefore, using Proposition~\ref{prop:condforsplitting}-(3), we see that $\mathcal{A}_I$ cannot split. The other direction follows by observing that there is only one nontrivial permutation in $S_2$ and again appealing to Proposition~\ref{prop:condforsplitting}.
\end{example}

\section{Uniformly Rigid actions}
\label{sec:morenoncommutative}
Most of the examples in our setup have a ``non-faithful"-characteristic. Precisely, the group acting minimally on the underlying space falls in the kernel of the action of the other component of the ambient tensor product. Motivated by the rigid topological dynamical systems studied in \cite{glasner1989rigidity}, we introduce the notion of uniformly rigid actions for general $C^*$-algebras. This allows us to give a class of examples satisfying the splitting theorem with the additional property that the action on the other component of the ambient tensor product is faithful.
\begin{definition}[Uniformly Rigid Actions]\label{def: strongly uniformly rigid}
Let $G\curvearrowright A$ be an action on a unital $C^*$-algebra $A$ by *-automorphisms. We say this action is uniformly rigid if there is a net of nontrivial elements $s_i\in G$ such that $\lim_i\left\|s_ia-a\right\|=0$ for all $a\in A$. In this case, we say that $\{s_i\}$ is a rigid sequence  witnessing the uniform rigidity for this action.
\end{definition}

\begin{proposition}\label{prop: ur is weaker than sur}
Let $G\curvearrowright X$ be a continuous action on a compact Hausdorff space $X$. Then $G\curvearrowright C(X)$ is uniformly rigid iff $G\curvearrowright X$ is uniformly rigid, i.e., there exists a sequence (or a net in the non-metrizable case) of nontrivial elements in $G$, say $\{s_i\}$ such that $s_i\rightarrow Id_X$ uniformly.
\end{proposition}
\begin{proof}
First, we show that $G\curvearrowright C(X)$ is uniformly rigid implies that $G\curvearrowright X$ is uniformly rigid. Let $\{s_i\}\subset G$ be a net witnessing the uniform rigidity for $G\curvearrowright C(X)$. Assume that $s_i$ does not converge to $\text{Id}_X$ uniformly, then there is a net $x_i\in X$ and a neighborhood $U\supset D$, the diagonal in $X\times X$ such that $(s_i^{-1}x_i,x_i)\not\in U$ for all $i$. Since $X$ is compact, we may assume that $\lim_ix_i=x$ and $\lim_is_i^{-1}x_i=y$. Note that $x\neq y$. Then there exists some $f\in C(X)$ with $0\leq f\leq 1$ and $f(x)=1$ and $f(y)=0$. Hence, as $i\to\infty$, we get the following contradiction:
\begin{align*}
0\leftarrow \left\|s_if-f\right\|=\sup_{x\in X}|f(s_i^{-1}x)-f(x)|\geq \sup_i|f(s_i^{-1}x_i)-f(x_i)|
\rightarrow |f(x)-f(y)|=1.
\end{align*}

To prove the converse, we fix any $\{s_i\}\subset G$ witnessing the uniform rigidity for $G\curvearrowright X$. Fix any $f\in C(X)$. Assume that $\lim_i\|s_if-f\|\neq 0$, then after passing to a subnet, we may assume there is some $\epsilon>0$ and points $x_i\in X$ such that $|f(s_i^{-1}x_i)-f(x_i)|\geq \epsilon$ for all $i$. Moreover, since $X$ is compact, we may assume $\lim_ix_i=x\in X$. Since $s_i^{-1}$ converges to $Id_X$ uniformly, we have $\lim_is_i^{-1}x_i=x$. Then, we get a contradiction by taking the limit on the above inequality using $f\in C(X)$.
\end{proof}
We now give some examples of uniformly rigid actions on noncommutative $C^*$-algebras. In fact, in these examples, the actions are strongly uniformly rigid in the sense that $\lim_i\sup_{\left\|a\right\|=1}\left\|s_ia-a\right\|=0$ for some net of nontrivial elements $s_i\in G$.
\begin{proposition}\label{prop: actions on M_n(C) are ur}
Let $G\curvearrowright M_n(\mathbb{C})$ be any action of an infinite discrete group by $*$-automorphisms. Then, this action is uniformly rigid.
\end{proposition}
\begin{proof}
Since every automorphism of $M_n(\mathbb{C})$ is inner, there is some map $\pi: G\rightarrow \mathcal{U}(n)$ such that the action is given by $Ad(\pi_g)$. Note that since $M_n(\mathbb{C})$ has a trivial center, it is not hard to see that $\pi_{st}\equiv \pi_s\pi_t~mod~\mathbb{T}$ and $\pi_{s^{-1}}\equiv \pi_{s^*}$, i.e., $\pi_{st}(\pi_s\pi_t)^*, \pi_{s^*}\pi_{s^{-1}}^*\in\mathbb{T}$ for all $s,t\in G$, where $\mathbb{T}$ is understood as the set of unitary scalar matrices in $\mathcal{U}(n)$. To show the action is strongly uniformly rigid, it suffices to show there is some nontrivial net $\{s_i\}\subset G$ such that $\lim_i\left\|\pi_{s_i}-\lambda I_n\right\|=0$ for some $\lambda\in \mathbb{T}$. Indeed, notice that
\begin{align*}
    \sup_{\left\|a\right\|=1}\left\|s_ia-a\right\|&=\sup_{\left\|a\right\|=1}\left\|\pi_{s_i}a-a\pi_{s_i}\right\|\\&\leq \sup_{\left\|a\right\|=1}\left\|\pi_{s_i}a-\lambda a\right\|+\sup_{\left\|a\right\|=1}\left\|\lambda a-a\pi_{s_i}\right\|\\&\leq 2\left\|\pi_{s_i}-\lambda I_n\right\|.
\end{align*}
Since $\mathcal{U}(n)$ is compact, we can find some $\{t_i\}\subset G$ such that $\lim_i\pi_{t_i}=u\in\mathcal{U}(n)$, then after passing to a subnet, we may assume that $\left\|\pi_{t_i}-u\right\|<\frac{1}{i}$ for all $i$. Then $\left\|\pi_{t_{i+1}}^*\pi_{t_i}-I_n\right\|=\left\|\pi_{t_i}-\pi_{t_{i+1}}\right\|\leq \frac{1}{i}+\frac{1}{i+1}\rightarrow 0$. We may write $\pi_{t_{i+1}^{-1}t_i}=\pi_{t_{i+1}}^*\pi_{t_i}\lambda_i$ for some $\lambda_i\in \mathbb{T}$. By passing to a subnet, we may assume that $\lim_i\lambda_i=\lambda\in\mathbb{T}$. Hence, we deduce that $\lim_i\left\|\pi_{s_i}-\lambda I_n\right\|=0$, where we set $s_i=t_{i+1}^{-1}t_i$.
\end{proof}

In fact, the proof can be generalized to deal with actions on finite-dimensional $C^*$-algebras.

\begin{proposition}\label{prop: actions on finite dimensional algebras are ur}
Let $G\curvearrowright A$ be an action  of an infinite discrete group on a finite-dimensional $C^*$-algebra $A$ by *-automorphisms. Then, this action is uniformly rigid.
\end{proposition}
\begin{proof}
By structure theorem for finite-dimensional $C^*$-algebras,
\[A\cong \oplus_{j=1}^mM_{n_j}(\mathbb{C})^{\oplus k_j},\]
where $m\geq 1$, $n_1<n_2<\ldots<n_m$ and $k_j\geq 1$ for all $1\leq j\leq m$. Below, we identify $A$ with $\oplus_{j=1}^mM_{n_j}^{\oplus k_j}(\mathbb{C})$.
Note that by considering central projections in $A$, we know that every *-automorphism of $A$ fixes each $j$-th summand $M_{n_j}(\mathbb{C})^{\oplus k_j}$ for all $1\leq j\leq m$.
Next, write the center $Z(A)=C(X)$ for some finite set $X$ and consider the action $G\curvearrowright X$. Since $X$ is finite, the kernel of this action is infinite and so without loss of generality, we may assume $G$ acts on $X$ trivially. Hence, each copy of $M_{n_j}(\mathbb{C})$ in the decomposition of $A$ is globally fixed.
In other words,  we may assume $G\curvearrowright A$ is given by a map $\pi: G\rightarrow \oplus_{j=1}^m\mathcal{U}(n_j)^{\oplus k_j}$. Note that we may embed $A=\oplus_{j=1}^mM_{n_j}(\mathbb{C})^{\oplus k_j}$ diagonally into $M_{\sum_{j=1}^mn_jk_j}(\mathbb{C})$. Observe that the map $\pi: G\rightarrow \oplus_{j=1}^m\mathcal{U}(n_j)^{\oplus k_j}$ also canonically becomes a map $\pi: G\rightarrow \mathcal{U}(\sum_{j=1}^mn_jk_j)$.
Then by applying the proof of Proposition \ref{prop: actions on M_n(C) are ur}, we would get a net $\{s_j\}\subset G$ witnessing the uniform rigidity for the action $\pi: G\curvearrowright M_{\sum_{j=1}^mn_jk_j}(\mathbb{C})$. This would entail a net witnessing the uniform rigidity for $G\curvearrowright A$.
\end{proof}

Similar to the commutative setting \cite{AGG}, we can introduce the so-called RSP sequence and get the following theorem.
\begin{thm}
Let $\beta: G\curvearrowright \mathcal{A}$ and $\alpha: G\curvearrowright Y$ be two continuous actions, where $\mathcal{A}$ is a $C^*$-algebra and $Y$ is a compact Hausdorff space.
Assume the following conditions hold true:
\begin{itemize}
    \item $\alpha: G\curvearrowright Y$ is a $G$-boundary, i.e., it is a strongly proximal minimal $G$-action;
    \item $\beta: G\curvearrowright \mathcal{A}$ is a uniformly rigid action;
    \item There exists an $(\alpha,\beta)$-RSP sequence $\{s_i\}\subset G$ in the sense that $\beta_{s_i}$ is a rigid sequence witnessing the uniform rigidity for $\beta$ and there exists some point $y^*\in Y$ such that $(s_i)_*\nu$ converges to $\delta_{y^*}$ weakly for all $\nu\in \text{Prob}(Y)$.
\end{itemize}
Then every $G$-invariant intermediate $C^*$-algebra $Q$ with $C(Y)\subset Q\subset C(Y)\otimes_{\text{min}} \mathcal{A}$ splits as a tensor product.
\end{thm}
\begin{proof}
If $\mathcal{A}$ is abelian, then this follows from Theorem \ref{thm: spliting using RSP sequence in the commutative times commutative setting}
 in view of Proposition \ref{prop: ur is weaker than sur}. Fix a general unital $C^*$-algebra $\mathcal{A}$ such that $G\curvearrowright Y$ and $G\curvearrowright\mathcal{A}$ satisfy the three conditions.
We claim that the induced $G$-action on $S(\mathcal{A})$, the state space of $\mathcal{A}$ equipped with the weak *-topology is uniformly rigid. Since $G\curvearrowright \mathcal{A}$ is uniformly rigid, we have rigid sequence $\{s_i\}\subset G$ implementing the rigid action. Let $\varphi\in S(\mathcal{A})$ and $\epsilon>0$ be given. Consider a basic open neighborhood $U_{\{x_1,x_2,\ldots,x_n\}}^{\epsilon}$ containing $\varphi$, where
\[U_{\{a_1,a_2,\ldots,a_n\}}^{\epsilon}=\left\{\psi\in S(\mathcal{A}): |\varphi(a_i)-\psi(a_i)|<\epsilon,~i=1,2,\ldots,n\right\}.\]
We can find $i_0$ such that for all $i\ge i_0$, we have $$\sup_{1\leq j\leq n}\|a_j-s_i^{-1}a_j\|=\sup_{1\leq j\leq n}\|s_ia_j-a_j\|<\epsilon.$$ 
Now, for all $i\ge i_0$, we see that 
\[|s_i.\varphi(a_j)-\varphi(a_j)|=|\varphi(s_i^{-1}a_j)-\varphi(a_j)|\le \|s_i^{-1}a_j-a_j\|<\epsilon,~j=1,2,\ldots,n.\]
Consequently, it follows that $s_i.\varphi\in U_{\{a_1,a_2,\ldots,a_n\}}^{\epsilon}$ for all $i\ge i_0$. Since the choice of $i_0$ is independent of the points in the neighborhood $U_{\{x_1,x_2,\ldots,x_n\}}^{\epsilon}$ or the state $\varphi$, we see that the action $G\curvearrowright S(\mathcal{A})$ is uniformly rigid. Consequently, the actions $G\curvearrowright Y$ and $G\curvearrowright S(\mathcal{A})$ satisfy the abovementioned assumptions. We can now use Theorem~\ref{thm: spliting using RSP sequence in the commutative times commutative setting} to conclude that every $G$-intermediate factor  $Z$ between $Y\times S(\mathcal{A})$ and $Y$ splits as a product $Y\times B$ for some $G$-factor $B$ of $S(\mathcal{A})$. Appealing to Theorem~\ref{thm:bridgecnc}, the proof is finished. 
\end{proof}
\begin{remark} 
Our original proof of the above theorem needs a stronger assumption on $G\curvearrowright \mathcal{A}$. We are very grateful to Hanfeng Li for pointing out that uniform rigidity is enough, which motivates us to isolate Theorem ~\ref{thm:bridgecnc}.
\end{remark}

Next, we present examples that fit into the above theorem. First, note that we have examples of strongly uniformly rigid actions on non-commutative $C^*$-algebras by either Proposition \ref{prop: actions on M_n(C) are ur} or Proposition \ref{prop: actions on finite dimensional algebras are ur}.

Our first example comes from the convergence groups. A group $\Gamma$ is called a convergence group when it admits an action $\Gamma\curvearrowright X$ with the following property. Given any distinct sequence of elements $\{g_n\}$ within $\Gamma$, there are two points $a,b\in X$ along with a subsequence $g_{n_k}$ such that the action of $g_{n_k}$ on $X\setminus\{b\}$ converges uniformly to $a$ on every compact subset of $X\setminus \{b\}$. We refer the readers to \cite{Tukia} for more details and properties of such groups.
\begin{example}
    Let $\Gamma$ be a non-elementary convergence group. Let $\beta: \Gamma\curvearrowright \mathcal{A}$ be a uniformly rigid continuous action and $\alpha: \Gamma\curvearrowright Y$ be the action on the limit set of a non-elementary convergence action of $G$. Then there exists an $(\alpha,\beta)$-RSP sequence $\{s_i\}\subset \Gamma$.
\end{example}
\begin{proof}
    The proof of \cite[Theorem 5.2]{AGG} can be applied here. Indeed, take any rigid sequence in $\Gamma$, say $\{\gamma_n\}$ witnessing the strongly uniform rigidity for $G\curvearrowright \mathcal{A}$. Then, that proof goes through once we know two facts about this sequence.
    
    (1) A conjugate of $\{\gamma_n\}$, say $\{g\gamma_n g^{-1}\}$, where $g\in \Gamma$ is still a rigid sequence witnessing the strongly uniform rigidity for $\Gamma\curvearrowright\mathcal{A}$.
    
    (2) A product of two rigid sequences is still a rigid sequence witnessing the strongly uniform rigidity for $\Gamma\curvearrowright\mathcal{A}$.

These two facts hold since *-automorphisms of $C^*$-algebras are isometries.
\end{proof}
For the same reason as above, it is also clear that the proof of \cite[Theorem 6.3]{AGG} shows the following holds.
\begin{example}
        Let $\Gamma<G=SL_d(\mathbb{R})$, $d\geq 2$, be a lattice. Let $\beta: \Gamma\curvearrowright \mathcal{A}$ be a strongly uniformly rigid action on a $C^*$-algebra $\mathcal{A}$ with a fixed rigid sequence $\{s_n: n\in\mathbb{N}\}$. Then there exists some $0\leq\ell\leq  n-1$ such that if we set  $\alpha: \Gamma\curvearrowright Y$ be the action on $Y=Gr(\ell, \mathbb{R}^d)$, the Grassmann variety of $\ell$-dimensional subspaces of $\mathbb{R}^d$ induced by the projective linear transformations of $\Gamma$ on $\mathbb{P}^{d-1}(\mathbb{R})$, the projective space of all lines in $\mathbb{R}^d$. Then there exists an $(\alpha, \beta)$-RSP sequence $\{s_i\}\subset \Gamma$.
\end{example}
\subsection{More examples in the commutative setup}
We first prove the following general theorem by mimicking the argument of \cite[Theorem 5.2]{AGG}.

\begin{thm}\label{thm: spliting using RSP sequence in the commutative times commutative setting}
Let $G\curvearrowright X$ and $G\curvearrowright Y$ be two continuous actions.
Assume the following conditions hold true:
\begin{itemize}
    \item $\alpha: G\curvearrowright Y$ is a $G$-boundary, i.e., it is a strongly proximal minimal $G$-action;
    \item $\beta: G\curvearrowright X$ is a uniformly rigid action;
    \item There exists an $(\alpha,\beta)$-RSP sequence $\{s_i\}\subset G$.
\end{itemize}
Then every $G$-invariant intermediate $C^*$-algebra $Q$ with $C(Y)\subset Q\subset C(X)\otimes_{\text{min}} C(Y)$ splits as a tensor product.
\end{thm}
\begin{proof}
Let $x, x'\in X$, $y_0\in Y$ be any points. Assume that $q(x, y_0)=q(x', y_0)$ for all $q\in Q$. By Proposition \ref{prop: equivalence relation formulation on splitting}, to finish the proof,
we just need to show that $q(x, y)=q(x', y)$ for all $y\in Y$.

Fix any $q\in Q$ and $y\in Y$.
Let $\{s_i\}\subset G$ be any uniformly rigid sequence on $X$. By assumption, there exists some $a\in Y$ such that $\lim_is_iz=a$ for all $z\in Y$. %\textcolor{red}{Should $r$ be $a$?}

Denote by $N(a, U)=\{s\in G:~ sa\in U\}$. Note that the action $G\curvearrowright Y$ is minimal is equivalent to say $N(a, U)$ is a syndetic set in $G$ for any $y\in Y$ and non-empty open set $U\subset Y$, i.e., there exists some finite subset $F\subset G$ with $G=FN(a, U)=\{st: ~s\in F, ~ t\in N(a, U)\}$. In particular, $N(a, U)$ is non-empty.

Given any open neighborhood $U\ni y$, take any $g\in N(a, U)$. Then set $t_i=gs_ig^{-1}$. Note that as a conjugation of $\{s_i\}$, $\{t_i\}$ is still a uniformly rigid sequence on $X$.

Notice that $\lim\limits_it_iy_0=ga$. Then we have
\begin{align*}
    q(x, ga)&=\lim_iq(t_ix, t_iy_0)=\lim_i(t_i^{-1}q)(x, y_0)\\
    &=\lim_i(t_i^{-1}q)(x', y_0)=\lim_iq(t_ix', t_iy_0)=q(x', ga).
\end{align*}
Now, by shrinking $U$, we get a sequence $\{g_Ua: ~U\ni y\}$ with $g_Ua\rightarrow y$ and $q(x, g_Ua)=q(x', g_Ua)$ from the above proof. Hence, after taking a limit along this sequence, we get that $q(x, y)=q(x', y)$.
\end{proof}

Below we present two classes of examples such that Theorem \ref{thm: spliting using RSP sequence in the commutative times commutative setting} could be applied. 

\begin{example}[Non-elementary Convergence groups]
Let $G$ be a non-elementary convergence group. Let $G\curvearrowright X$ be a uniformly rigid continuous action and $G\curvearrowright Y$ be the limit set for a non-elementary convergence action of $G$. Note that the action $G\curvearrowright Y$ is topologically free since every $e \ne \gamma \in \Gamma$  has at most two fixed points.

Let $1 \ne \gamma_n \in G$ be any uniformly rigid sequence on $X$. Since $G\curvearrowright Y$ is a convergence action, there is a subsequence $\{n_i\} \subset \Z$ and $a,r \in Y$ such that $\lim_{i \rightarrow \infty}\gamma_{n_i} z = a$ for every $z \in Y \setminus \{r\}$ and the convergence is uniform on compact subsets of $Y \setminus \{r\}$. Upon passing to a subsequence (which by abuse of notation is still denoted by $\gamma_{n_i}$), we can assume in addition that $\lim_{i \rightarrow \infty} \gamma_{n_i} r = v$ for some $v \in Y$. By the minimality of the action, there is an element $g \in G$ such that $g^{-1}r \not \in \{a,v\}$. By arguing similarly as in \cite[Theorem~5.2]{AGG}, it follows that sequence $\{s_i = \gamma_{n_i} g \gamma_{n_i} g^{-1}\}$ $(\alpha,\beta)$-RSP sequence for $G$. Consequently, all the conditions of Theorem~\ref{thm: spliting using RSP sequence in the commutative times commutative setting} are met. Therefore, every  $G$-invariant intermediate $C^*$-algebra $Q$ with $$C(Y)\subset Q\subset C(X)\otimes_{\text{min}} C(Y)$$ splits as a tensor product.
\end{example}

We now give an example of a group that does not belong to the convergence groups.
\begin{example}[Higher rank lattices]
Let $\Gamma<G=SL_d(\mathbb{R})$, $d\geq 2$, be a lattice with trivial center. Let $\Gamma\curvearrowright X$ be a uniformly rigid action with a fixed, rigid sequence $\{s_n: n\in\mathbb{N}\}$. Then there exists some $0\leq\ell\leq  n-1$ such that if we set  $\Gamma\curvearrowright Y$ be the action on $Y=Gr(\ell, \mathbb{R}^d)$, the Grassmann variety of $\ell$-dimensional subspaces of $\mathbb{R}^d$ induced by the projective linear transformations of $\Gamma$ on $\mathbb{P}^{d-1}(\mathbb{R})$, the projective space of all lines in $\mathbb{R}^d$. It follows from the proof of \cite[Theorem~6.3]{AGG} that the action $\Gamma\curvearrowright Y$ is a strongly proximal minimal action. Moreover, arguing similarly as in the proof of \cite[Theorem~6.3]{AGG}, we see that there is an $(\alpha,\beta)$-RSP sequence for $\Gamma$. Therefore, using Theorem~\ref{thm: spliting using RSP sequence in the commutative times commutative setting}, we see that every $G$-invariant intermediate $C^*$-subalgebra $Q$ between $C(Y)$ and $C(X)\otimes_{\text{min}} C(Y)$ splits as a tensor product. 
\end{example}

\end{document}